\RequirePackage{etoolbox}
\csdef{input@path}{{style/}{graphics/}}
\documentclass[aap,MSNbibl,seceqn,dvips]{arximspdf}
\makeatletter
   \@ifpackageloaded{graphicx}{}{\usepackage{graphicx}}
\makeatother

\doi{10.1214/14-AAP1021} 
\volume{25}
\issue{3}
\pubyear{2015}
\firstpage{1232}
\lastpage{1278}

\makeatletter
\newcommand{\widebar}{\overline}
\newcommand{\rrvert}{\vert}
\newcommand{\llvert}{\vert}
\newtheorem{thma}{Theorem}[section]
\newtheorem{lemma}[thma]{Lemma}
\newtheorem{prop}[thma]{Propostion}
\newtheorem{claim}{Claim}
\newproclaim{remark}{Remark}
\newcommand{\Ai}{\mathrm{Ai}}
\newcommand{\oAi}{\operatorname{Ai}}
\makeatother

\begin{document}
\begin{frontmatter}

\title{Asymptotic domino statistics in the Aztec diamond}
\runtitle{Asymptotics in the Aztec diamond}

\begin{aug}
\author[A]{\fnms{Sunil}~\snm{Chhita}\corref{}\thanksref{T1}\ead[label=e1]{schhita@iam.uni-bonn.de}\ead[label=e11]{chhita@kth.se}},
\author[B]{\fnms{Kurt}~\snm{Johansson}\ead[label=e2]{kurtj@kth.se}\thanksref{T1,T2}}
\and
\author[C]{\fnms{Benjamin}~\snm{Young}\thanksref{T1}\ead[label=e3]{bjy@uoregon.edu}}
\runauthor{S. Chhita, K. Johansson and B. Young}
\affiliation{University of Bonn, Royal Institute of
Technology (KTH)\\ and University of Oregon}
\address[A]{S. Chhita\\
Institute for Applied Mathematics\\
University of Bonn\\
Endenicher Allee 60\\
D-53115 Bonn\\
Germany\\
\printead{e1}\\
\phantom{E-mail: }\printead*{e11}}
\address[B]{K. Johansson\\
Department of Mathematics\\
Royal Institute of Technology (KTH)\\
SE-100 44 Stockholm\\
Sweden\\
\printead{e2}} 
\address[C]{B. Young\\
Department of Mathematics\\
University of Oregon\\
Eugene, Oregon 97403\\
USA\\
\printead{e3}}
\end{aug}
\thankstext{T1}{Supported by the Knut and Alice Wallenberg Foundation
Grant KAW:2010.0063.}
\thankstext{T2}{Supported by the Swedish Research Council (VR).} 

\received{\smonth{1} \syear{2013}}
\revised{\smonth{1} \syear{2014}}

%
\begin{abstract}
We study random domino tilings of the Aztec diamond with different
weights for horizontal and vertical dominoes. A domino tiling of an
Aztec diamond can also be described by a particle system which is a
determinantal process. We give a relation between the correlation
kernel for this process and the inverse Kasteleyn matrix of the Aztec
diamond. This gives a formula for the inverse Kasteleyn matrix which
generalizes a result of Helfgott. As an application, we investigate the
asymptotics of the process formed by the southern dominoes close to the
frozen boundary. We find that at the northern boundary, the southern
domino process converges to a thinned Airy point process. At the
southern boundary, the process of holes of the southern domino process
converges to a multiple point process that we call the thickened Airy
point process. We also study the convergence of the domino process in
the unfrozen region to the limiting Gibbs measure.
\end{abstract}

%
\begin{keyword}[class=AMS]
\kwd[Primary ]{60G55}
\kwd[; secondary ]{60C05}
\end{keyword}
\begin{keyword}
\kwd{Aztec diamond}
\kwd{domino tiling}
\kwd{dimer covering}
\kwd{determinantal point process}
\end{keyword}
\end{frontmatter}

\section{Introduction}\label{sec1}\label{secIntroduction}

The \emph{Aztec Diamond of order $n$} is a planar region which can be
completely tiled with \emph{dominoes}, two-by-one rectangles. Over the
past twenty years, this particular shape has come to occupy a central
place in the literature of domino tilings of plane regions. Tilings of
large Aztec diamonds exhibit striking features---the main one being
that these tilings exhibit a limit shape, described by the so-called
Arctic circle theorem~\cite{JPS98}.
See Figure~\ref{figprettypicturesofaztecdiamonds} for pictures of
tilings of a relatively large Aztec diamond.

%
\begin{figure}

\includegraphics{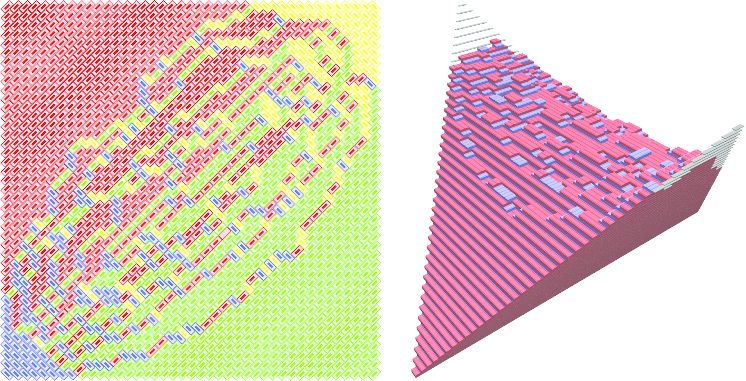}

\caption{Left:\vspace*{1pt} a tiling of an Aztec diamond of order $41$ with $a =
\frac{1}{2}$, and the corresponding dimer cover. The green dominoes
along each horizontal row give the southern domino process. Right: the
height function associated to this tiling (realized as a pile of
Levitov blocks \cite{Lev89,Lev90}).}\label{figprettypicturesofaztecdiamonds}
\end{figure}

There are several alternate descriptions of a tiling of an Aztec
diamond. A~domino tiling is equivalent to a perfect matching, or \emph
{dimer cover}, on the dual graph of the region which is tiled. There is
an equivalent family of nonintersecting lattice paths, called DR
paths~\cite{LRS01}, and there is a description as a stack of a certain
sort of blocks, called \emph{Levitov blocks} \cite{Lev89,Lev90};
see~Figure~\ref{figprettypicturesofaztecdiamonds}.
There is also a well-studied \emph{interlacing particle process} which
is equivalent\vspace*{1pt} to the tiling model in a certain sense, but this
equivalence is not bijective: there are $2^{n(n+1)/2}$ tilings of the
Aztec diamond, whereas the number of configurations of the particle
process are equinumerous with order-$n$ \emph{alternating sign
matrices} or configurations of the \emph{six-vertex model}. However,
the correspondence is \emph{weight preserving} and \emph{locally
defined}; it maps certain collection of tilings to a configuration of
the \emph{free-fermion} six-vertex model~\cite{EKLP92}, preserving
the relative weights.

It is this point process whose local asymptotics have been studied most
thoroughly~\cite{Joh05}; in particular, the boundary of the frozen
region is more easily described using the particles, since it is
related to the position of the last particles on different lines. There
is also a relationship between these particles and a certain sort of
\emph{zig-zag path} in the tiling (distinct from the zig-zag paths
studied in~\cite{Bro12}) which is helpful.
However, the many-to-one nature of the correspondence necessarily loses
some of the information about the original tilings.
In both the domino and particle pictures, we get determinantal point
processes, although the precise combinatorial relation between the two
processes is not immediate.

We correct this situation in this paper, by giving a formula for the
inverse Kasteleyn matrix for the Aztec diamond. This generalizes a
previous result by Helfgott~\cite{Hel00} to the case when the
horizontal and vertical dominoes have different weights. Using an
observation of Kenyon~\cite{Ken00}, it is then possible to compute
probabilities for various configurations of dominoes and their
asymptotic limits as the size of the Aztec diamond increases. In
particular, we find the behavior of the boundary of the frozen region
when we can only see one type of domino.


\subsection{The southern domino process: North boundary}\label{sec1.1}
There are in fact four different types of dominoes in a tiling: the
dominoes can be placed in two orientations, each of which comes in two
different parities (determined by the bipartition of the dual graph on
which the dominoes are placed). Due to the Arctic circle theorem, with
probability one, there are only dominoes of one of these four types
clustered near each of the four corners of the Aztec diamond. For this
reason (and others, see~\cite{EKLP92}) we call the four types of
dominoes \emph{north}, \emph{south}, \emph{east and west}, which are colored red,
green, yellow and blue in Figure~\ref{figprettypicturesofaztecdiamonds} in the electronic version of this article.

For the moment, ignore all but the southern dominoes (the green ones in
Figure~\ref{figprettypicturesofaztecdiamonds}). Viewing each
domino as a
\emph{point}, the set of all southern dominoes form a determinantal point
process. Note that the positions of the southern dominoes do not
specify the
tiling uniquely so they only give a partial description of the tiling (though,
together with any of the other types of dominoes, they do). We analyze the
distribution of the southern dominoes along a diagonal line in a large Aztec
diamond, scaled so as to study the two intersections between this line
and the
frozen boundary of the tiling (an ``arctic ellipse,'' in the weighted case).
It would be possible to extend what we have done to analyze the joint
distribution of all southern dominoes in the Aztec diamond; as is, our analysis
extends previous results in \cite{CEP96} on placement
probabilities of single dominoes.

We will show that the appropriate scaling limit of the point
process of southern dominoes along a diagonal line close to the
boundary of the \emph{northern} frozen region is given by a thinned
Airy kernel point process, where the amount of thinning depends on the
relative weight of the horizontal and vertical dominoes.
This can be heuristically understood in the following way. The Airy
kernel point process, as mentioned earlier, is the edge limit of the
particle process along a diagonal line, and these in turn are given by
the intersections of the nonintersecting paths and the diagonal line.
Sometimes these intersections occur along a southern domino, and
sometimes not. If we only see the southern dominoes we only see some of
these intersections, and which of them we see is essentially random; so
we might expect, in the limit, a random thinning of the Airy kernel
point process. A priori it is not clear that the thinning becomes
independent in the limit, but this turns out to be the case.

\subsection{The southern domino process: Southern boundary}\label{sec1.2}
If, instead, we examine the southern frozen region, we find that almost
all of the dominoes are coming from the southern domino process and
thus lie in a frozen, brickwork pattern predominantly. The southern
boundary is a ``hole'' in this regular pattern. Consider the holes
between the
southern dominoes along a diagonal line in a neighborhood of the
southern boundary. These holes also form a determinantal point process, but
in a scaling limit it does not converge to a \emph{simple} point
process, but rather to a \emph{multiple point process} with
independent geometric probabilities
for the multiple points in an Airy kernel point process. This can be
seen in Figure~\ref{figprettypicturesofaztecdiamonds}: there is a
tendency for dominoes of like types to cluster together along the
southern boundary. This tendency continues even in the limit, with a
cluster of $k$ dominoes becoming a point of multiplicity $k$. The
multiplicity increases as we go toward the lower tangency point of the
arctic ellipse, a fact which can also be observed in Figure~\ref
{figprettypicturesofaztecdiamonds}.

\subsection{Previous work on domino tilings}\label{sec1.3}
\label{subsecPreviouswork}
Domino tilings on the Aztec diamond were originally introduced in~\cite
{EKLP92,EKLP92a} as a model connected with the \emph{alternating
sign matrices}. In this section, we give only a partial overview of the
literature on the asymptotics of domino tilings of the Aztec diamond.

The limit shape for random tilings of the Aztec diamond, the so-called
Arctic circle theorem, was first computed for $a=1$ in~\cite{JPS98},
where $a$ is the weight of each vertical domino. Since then, there have
been a variety of different and interesting approaches to compute the
limit shape which hold for general $a$~\cite
{CEP96,Joh05,Rom12,KO07}. The existence of limit shapes is not
limited to domino tilings; limit shapes also exist for random lozenge
tilings, for example, the boxed plane partition~\cite{CLP98}. These
examples provided a motivation for a theory of the existence of limit
shapes for general tiling models on bipartite graphs \cite{CKP01,KO07}.

The edge behavior, that is, the behavior between the frozen and
unfrozen regions has, been of particular interest\vspace*{1pt} to the random matrix
community, as the fluctuations are of size $n^{1/3}$. This is the same
size as the fluctuations of the largest eigenvalue of the \emph
{Gaussian unitary ensemble} (\emph{GUE}). Indeed, \cite{Joh02,Joh05}
showed that the law of the particles associated to the tiling is given
by the \emph{Airy process} and that the position of the last particle
is given by the Tracy Widom distribution, $F_2$; see, for example,~\cite{AGZ10}. Furthermore, Johansson and Nordenstam \cite{JN06}
showed that the distribution of these particles becomes the \emph{GUE
minor process} at the intersection of the liquid region and the
boundary of the Aztec diamond while Fleming and Forrester \cite{FF11}
obtained similar results for a certain half Aztec diamond.

There are a handful of other explicitly inverted Kasteleyn matrices in
the literature for domino tilings. The inverse of the Kasteleyn matrix
of the Aztec diamond was computed by Helfgott~\cite{Hel00} in the
case of the uniform measure on domino tilings; the results in~\cite
{CKP01,KOS06} rely on explicit inverses of four Kasteleyn-like
matrices that, together, count perfect matchings on a torus-embedded
graph. Finally, Kasteleyn~\cite{Kas61} and independently Temperley
and Fisher in~\cite{TF61} compute the eigenvalues and eigenvectors of
the Kasteleyn matrix of the $m \times n$ grid graph explicitly.

A proof of convergence to the Gaussian Free Field, following~\cite
{BF08} should be possible. In fact, an earlier preprint of this
paper (dated December 21, 2012 and posted on the arXiv) stated such a claim
as Theorem 2.9 and outlined a proof, skipping over many details. We
would like to retract this theorem and its outlined proof, for the
following reason: the details that we omitted (largely estimates on~$K^{-1}$) were numerous enough and technical enough that even a rather
dedicated reader would have been hard-pressed to supply them all.
Moreover, in revision, we found it impossible to include enough
details of these estimates while keeping the discussion brief. The
proof, if and when it appears in the literature, will have to be in
its own paper. Instead, we include in Section~\ref{sec6} only a list of the
estimates that would be needed in order to demonstrate convergence.
We sincerely thank the anonymous referee for bringing this error to
our attention.

\section{Results}\label{sec2}\label{secResults}

In this section, we give the results of our paper and the necessary
prerequisites to understand these results. There are three types of results:
\begin{longlist}[(3)]
\item[(1)] the inverse of the Kasteleyn matrix (Section~\ref{subsecTheinverseoftheKasteleynmatrix}),
\item[(2)] results on the southern domino process close to the edges
of the unfrozen region (Section~\ref{subsecEdgefluctuationsofsoutherndominoes}),
\item[(3)] local Gibbs measure (Section~\ref{subsecBulk}).
\end{longlist}

\subsection{The inverse of the Kasteleyn matrix}\label{sec2.1}
\label{subsecTheinverseoftheKasteleynmatrix}

\subsubsection{Definitions}\label{sec2.1.1}
In this paper, the Aztec diamond is rotated by $\pi/4$ counter
clockwise from the convention set in \cite{EKLP92}.
Because there are many possibilities for coordinate systems of Aztec
diamonds, we will refer to our coordinate system as the \emph
{Kasteleyn coordinates}.
In the Kasteleyn coordinates, an \emph{Aztec diamond of order} $n$
consists of squares with corners $(k-1,l)$, $(k,l-1)$, $(k+1,l)$ and
$(k,l+1)$ for either $k\bmod2=1$ and $l \bmod2=0$ with $1 \leq k \leq
2n-1$ and $0 \leq l \leq2n$, or $k \bmod2=0$ and $l \bmod2=1$ with
$0 \leq k \leq2n$ and $1 \leq l \leq2n-1$.
A \emph{domino} is a union of two adjacent squares which share an edge.
A \emph{domino tiling} of the Aztec diamond is any arrangement of
dominoes such that each square of the Aztec diamond is covered exactly
once by a domino.

The dual graph of the Aztec diamond (without its external face) is a
bipartite graph which has vertices $\mathtt{W} \cup\mathtt{B}$ where
%
\begin{eqnarray}
\mathtt{W} &=&\bigl\{(x_1,x_2)\dvtx  x_1\bmod2=1,x_2\bmod2=0,
\nonumber\\[-8pt]\\[-8pt]
&&\hspace*{44pt}  1 \leq x_1 \leq2n-1, 0 \leq x_2
\leq2n \bigr\}\nonumber
\end{eqnarray}
and
%
\begin{eqnarray}
\mathtt{B} &=& \bigl\{(x_1,x_2)\dvtx  x_1 \bmod2=0,
x_2 \bmod2=1,
\nonumber\\[-8pt]\\[-8pt]
&&\hspace*{44pt} 0 \leq x_1 \leq2n, 1 \leq x_2
\leq2n-1 \bigr\},\nonumber
\end{eqnarray}
which correspond to the white and black vertices, respectively, written
in terms of the Kasteleyn coordinates.
To distinguish between the primal and dual graphs, we will refer to the
dual graph of the Aztec diamond as the \emph{Aztec diamond graph}. We
shall also set $e_1=(1,1)$ and $e_2=(-1,1)$.
The edge set of the Aztec diamond graph consists of all the edges
$(x,y)$ with $y-x \pm e_i$ for $i \in\{1,2\}$ for $x \in\mathtt{W}$
and $y \in\mathtt{B}$.

A domino on the dual graph is an edge which is called a \emph{dimer}.
A domino tiling on the dual graph is a subset of edges such that each
vertex is incident to exactly one edge.
This collection of edges is called a \emph{dimer covering}.
Domino tilings of the Aztec diamond are equivalent to dimer coverings
of the Aztec diamond graph.
See Figure~\ref{figdominoesdimers} for the Aztec diamond graph with
its coordinates and an example of a dimer covering.

For $b \in\mathtt{B}$ and $w \in\mathtt{W}$, we say that a dimer
$(b,w)$ is:
\begin{itemize}
\item a \emph{north} dimer if $w=b+e_1$,
\item an \emph{east} dimer if $w=b+e_2$,
\item a \emph{south} dimer if $w=b-e_1$,
\item a \emph{west} dimer if $w=b-e_2$.
\end{itemize}
There is a corresponding notion for dominoes, and this terminology
agrees with that introduced in~\cite{JPS98}. We will interchange
between dominoes and dimers.

%
\begin{figure}

\includegraphics{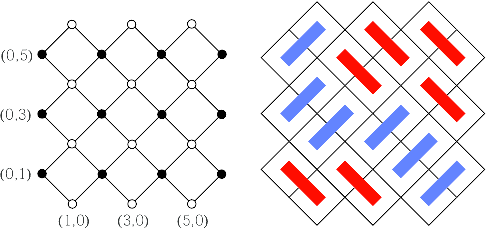}

\caption{The figure on the right shows the coordinates of the Aztec
diamond graph with the white and black vertices drawn in.
The figure on the right shows an Aztec diamond of size 3 with a dimer
covering of the dual graph. The domino tiling can be seen by placing
dominoes over the dimers.}\label{figdominoesdimers}
\end{figure}

\subsubsection{Determinantal point processes}\label{sec2.1.2}

Determinantal point processes are a key part of the analysis used in
this paper. Here, we briefly describe these processes but more in-depth
treatises of determinantal point processes can be found in \cite
{Joh06} and \cite{SOS00}.

Let $\Lambda$ be a Polish space, and take $\mathcal{M}(\Lambda)$ to
denote the space of counting measures $\xi$ on $\Lambda$ with $\xi
(B)<\infty$ for every bounded $B \subset\Lambda$. A point process on
$\Lambda$ is a probability measure $\mathbb{P}$ on $\mathcal
{M}(\Lambda)$. Let $M_n$ denote the factorial moment measure, that is,
for disjoint Borel sets $A_1,\ldots, A_m$ in $\Lambda$ and for all
$(n_1,\ldots,n_m) \in\mathbb{N}^m$
%
\begin{equation}
M_n\bigl(A_1^{n_1} \times\cdots\times
A_m^{n_m} \bigr) =\mathbb{E} \Biggl[ \prod
_{i=1}^m \frac{ \xi(A_i)!}{(\xi(A_i)-n_i)!} \Biggr].
\end{equation}

Suppose that $\lambda$ is a reference measure on $\Lambda$. For
example, if $\Lambda=\mathbb{R}$, we can choose $\lambda$ to be the
Lebesgue measure. If
%
\begin{equation}
M_n(A_1,\ldots,A_n)= \int
_{A_1 \times\cdots\times A_n} \rho_n(x_1,\ldots,x_n) \,d\lambda(x_1) \cdots d \lambda(x_n)
\end{equation}
for all Borel sets $A_i$ in $\Lambda$, we call $\rho_n$ to be the
$n$th correlation function. For discrete processes, $\rho_n(x_1,\ldots,x_n)$ is equal to the probability of $n$-tuples of particles at $x_1,\ldots, x_n$, whereas for continuous processes, $\rho_n$ is the density
of seeing particles. For example, if $\rho_n (x_1, \ldots, x_n)=\rho
(x_1)\cdots\rho(x_n)$ where $\rho\in L^1$ with $\Lambda=\mathbb{R}$
and $\lambda$ is the Lebesgue measure, then the point process is the
Poisson point process on $\mathbb{R}$.

A point process is called \emph{determinantal} if there exists a
function $\mathbb{K}\dvtx \Lambda\times\Lambda\to\mathbb{C}$ called
the correlation kernel, with
%
\begin{equation}
\rho_n(x_1, \ldots, x_n) = \det\bigl(
\mathbb{K}(x_i,x_j) \bigr)_{i,j=1}^n.
\end{equation}
This leads to the following characterization of a determinantal point
process. Let $C_c^+(\Lambda)$ be the set of all
nonnegative continuous functions
on $\Lambda$ with compact support. Take $\psi\in C_c^+(\Lambda)$,
let $A$ denote the support of $\psi$ and set $\phi=1-e^{-\psi}$. Let
$\mathbb{I}_A$ denote the indicator
function for the set $A$. Then,
provided $\phi\mathbb{K}\mathbb{I}_A$ is trace class, and the
Fredholm determinant is given by its Fredholm expansion,
%
\begin{equation}
\mathbb{E} \bigl[ e^{-\sum_j\psi(x_j)} \bigr]=\det(I-\phi\mathbb{K}
\mathbb{I}_A)_{L^2(\Lambda,\lambda)},
\end{equation}
where ${x_j}$ are the points in the process.

\subsubsection{Particles}\label{sec2.1.3} \label{subsubParticles}

Another way of viewing domino tilings of an Aztec diamond is a particle
system formed from the zig-zag particles used in~\cite{Joh05}. These
particles can be described as follows:
for $w \in\mathtt{W}$, we have a blue particle at $w$ if and only if
a dimer covers the edge $(w+e_1,w)$ or the edge $(w-e_2,w)$.
For $b \in\mathtt{B}$, we have a red particle at $b$ if and only if a
dimer covers the edge $(b,b-e_1)$ or the edge $(b,b-e_2)$. By this
setup, particles are present on south and west dimers, with blue
particles sitting on white vertices and red particles sitting on black vertices.

%
\begin{figure}[t]

\includegraphics{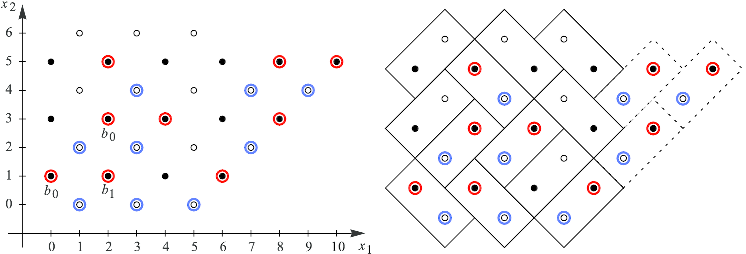}

\caption{The figure on the left shows the red--blue particles with the
Kasteleyn orientation for an Aztec diamond of size 3 (with additional
vertices). The figure on the right shows the same configuration of
red--blue particles with the domino tiling. This includes the three
additional south dominoes.}\label{figdominoparticles}
\end{figure}

The particle system considered in~\cite{Joh05} came with its own
coordinate system; see Figure~4 in Section~2 of that paper. The
transformation between that system of coordinates and the Kasteleyn
coordinates is
%
\begin{equation}
\label{dimerstoparticles1} \cases{ u_1=x_2,
\vspace*{3pt}\cr
\displaystyle
\displaystyle u_2=\frac{x_2-x_1+1}{2},}
\end{equation}
where $(u_1,u_2)$ are the particle coordinates and $(x_1,x_2)$ are the
Kasteleyn coordinates.
Figure~\ref{figdominoparticles} shows the red--blue particles along
with the corresponding tiling.

It is shown in \cite{Joh05} that the particles form a determinantal
point process with correlation kernel given by
%
\begin{equation}\label{dimerstoparticlestwo}
K_n(u_1,u_2;  v_1,v_2)=
\widetilde{K}_n (u_1,u_2; v_1,v_2)
- \phi_{u_1,v_1} (u_2,v_2),
\end{equation}
where
%
\begin{eqnarray}
&& \widetilde{K}_n(2r-\varepsilon_1, u_2; 2s-
\varepsilon_2,v_2)
\nonumber\\[-8pt]\label{dimerstoparticlesKtilde} \\[-8pt]
&&\qquad = \frac{1}{(2\pi i)^2} \int
_{\gamma_{r_1}} \frac{dw}{w} \int_{\gamma
_{r_2}}\frac{dz}{z} \frac{z^{v_2}}{w^{u_2}} \frac{(1-
az)^{n-s+\varepsilon_2} (1+a/z)^s}{(1- a w)^{n-r+\varepsilon_1}
(1+a/w)^r}\frac{w}{w-z},\nonumber
\\
\quad\qquad && \phi_{2r-\varepsilon_1,2s-\varepsilon_2} (u_2,v_2)
\nonumber\\[-8pt]\label{dimerstoparticlesfour} \\[-8pt]
&&\qquad = \frac{\mathbb
{I}(2r-\varepsilon_1 < 2s - \varepsilon_2)}{2\pi i} \int
_{\gamma_1} z^{v_2-u_2} \frac{ (1- az)^{r-s+\varepsilon_2-\varepsilon
_1}}{(1+a/z)^{r-s}} \frac{d z}{z}\nonumber
\end{eqnarray}
and $\varepsilon_1,\varepsilon_2 \in\{0,1\}$, $a<r_1<1/a$,
$0<r_2<r_1$ and $\gamma_t$ denotes a circle around $0$ with radius $t$.
Before setting $a=1$, one has to make an appropriate deformation of contours.

\subsubsection{The Kasteleyn matrix and Kenyon's formula}\label{sec2.1.4}
\label{subsecTheKasteleynmatrix}

The Kasteleyn matrix, introduced in \cite{Kas61,Kas63}, can be used
to count the number of weighted dimer coverings of a graph, and the
inverse of the Kasteleyn matrix can be used to compute local
statistics~\cite{Ken97}.

For a finite bipartite graph $G$, a \emph{Kasteleyn matrix} is a
signed weighted adjacency matrix of the graph with rows indexed by
black vertices and columns indexed by white vertices; see \cite
{Ken09} for details.
The sign is chosen according to a \emph{Kasteleyn orientation} of the
graph. This means assigning a sign (possibly complex valued) to each
edge weight so that the product of the edge weights around each face is
negative.

For the Aztec diamond graph, we denote $K$ to be the matrix with $K\dvtx
\mathtt{B} \times\mathtt{W} \to\mathbb{C}$ where $K(b,w)=K_{b,w}$
for $b=(x_1,x_2) \in\mathtt{B}$ and $w \in\mathtt{W}$ with
%
\begin{equation}
\label{kasteleynweights}
\qquad K (b,w)= \cases{ (-1)^{l+(x_1+x_2-1)/2},
&\quad if $w=b+
(-1)^l e_1 \in\mathtt{W}$,
\vspace*{3pt}\cr
(-1)^{l+(x_1+x_2-1)/2} a i, &
\quad if $w=b-(-1)^le_2 \in\mathtt{W}$,
\vspace*{3pt}\cr
0, &\quad
otherwise.}
\end{equation}
This matrix is a \emph{Kasteleyn matrix} for the Aztec diamond graph.
This matrix is a~Kasteleyn matrix for the Aztec diamond graph; see Figure \ref{figweights}.

%
\begin{thma}[(\cite{Kas61})]
The number of weighted dimer coverings of the Aztec diamond graph is
equal to $|\det K|$.
\end{thma}

%
\begin{figure}[b]

\includegraphics{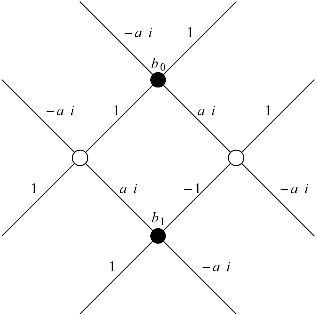}

\caption{The complex weights associated to the Kasteleyn matrix given
in~(\protect\ref{kasteleynweights}). We have that the vertex
$b_i=(x_1,x_2)$ has $(x_1+x_2-1)/2 \operatorname{mod}2 =i$ for $i \in
\{0,1\}$.}\label{figweights}
\end{figure}

In~\cite{Ken97}, Kenyon found that the dimers form a determinantal
point process with the correlation kernel written in terms of the
inverse of the Kasteleyn matrix (referred to as the \emph{inverse
Kasteleyn matrix}).
Here, we state that result for the Kasteleyn matrix given in~(\ref
{kasteleynweights}).
Suppose that $E=\{e_i\}_{i=1}^n$ are a collection of distinct edges
with $e_i=(b_i,w_i)$, where $b_i$ and $w_i$ denote black and white vertices.

\begin{thma} \cite{Ken97}\label{localstatisticsthm}.
The dimers form a determinantal point process on the edges of the Aztec
diamond graph with correlation kernel given by
%
\begin{equation}
L(e_i,e_j) = K(b_i,w_i)
K^{-1} (w_j,b_i),
\end{equation}
where $K(b,w) = K_{bw}$ and $K^{-1}(w,b)= (K^{-1})_{wb}$.
\end{thma}

The above formula is sometimes referred to as \emph{Kenyon's formula}.

\begin{pf*}{Proof of Theorem \ref{localstatisticsthm}}
By \cite{Ken97}, we have that the probability of finding dimers at
the edges $e_1,\ldots, e_n$ is
%
\begin{eqnarray}
\label{localstats} \mathbb{P}(e_1, \ldots, e_n) & =& \prod
_{i=1}^n K(b_i,w_i)
\det\bigl( K^{-1}(w_j,b_i)
\bigr)_{i,j=1}^n
\nonumber
\\
& =& \det\bigl(K(b_i,w_i) K^{-1}(w_j,b_i)
\bigr)_{i,j=1}^n
\\
& =& \det\bigl( L(e_i,e_j) \bigr)_{i,j=1}^n.
\nonumber
\end{eqnarray}\upqed
\end{pf*}

\subsubsection{The inverse Kasteleyn matrix}\label{sec2.1.5}

The inverse Kasteleyn matrix for domino tilings of the Aztec diamond
was originally computed in \cite{Hel00} for the case when $a=1$.
In that paper, Helfgott explicitly enumerated $K^{-1}(w,b)\* 2^{n(n+1)/2}$
which is the number of signed dimer coverings of an Aztec diamond graph
with the vertices $w \in\mathtt{W}$ and $b \in\mathtt{B}$ removed.
We generalize this formula so that we can consider different weights
for vertical and horizontal tiles:


%
\begin{thma} \label{discretecouplingfunction}
For $x=(x_1,x_2) \in\mathtt{W}$ and $y=(y_1,y_2) \in\mathtt{B}$, we have
%
\begin{eqnarray}
K^{-1}(x,y) &=& \cases{ f_1(x,y), &\quad for
$x_1<y_1+1$,
\vspace*{3pt}\cr
f_1(x,y) -
f_2(x,y), &\quad for $x_1 \geq y_1+1$,}
\end{eqnarray}
where
%
\begin{eqnarray}
\label{functionf1} f_1(x,y) & =& \frac{(-1)^{(y_1+y_2+x_1+x_2)/4
}}{(2\pi i)^2}\nonumber
\\
&&{}\times \int
_{\mathcal{E}_2} \int_{\mathcal{E}_1} \frac{w^{y_1/2}}{z^{(x_1+1)/2}(w-z)}
\\
&&\hspace*{41pt}{}\times \frac{ (a+z)^{x_2/2} (a
z-1)^{(2n-x_2)/2}}{ (a w-1)^{(2n+1-y_2)/2} (a+w)^{(y_2+1)/2}} \,dz \,dw
\nonumber
\end{eqnarray}
and
%
\begin{eqnarray}\label{functionf2}
f_2(x,y)&=&\frac{(-1)^{(x_1+ x_2 + y_1 +
y_2)/4}}{2\pi i} a^{(y_2-x_2-1)/2}
\nonumber\\[-8pt]\\[-8pt]
&&{}\times  \int
_{\mathcal{E}_1} \frac{ z^{(y_2-x_2-1)/2} (1/a
+z)^{(y_1-x_1-1)/2}}{ (1/a+a+z)^{(y_2-x_2+1)/2}}\,dz,\nonumber
\end{eqnarray}
where $\mathcal{E}_1$ is the positively oriented contour
$|z|={\epsilon}$, $\mathcal{E}_2$ is the positively oriented contour
$|w-1/a|={\epsilon}$ and the contours do not intersect.
\end{thma}

For domino tilings, the inverse Kasteleyn matrix cannot be obtained
directly from the correlation kernel of the red--blue particles
introduced in Section~\ref{subsubParticles}. However, the correlation
kernel of the red--blue particles can be directly obtained from the
inverse Kasteleyn matrix.
This is different to lozenge tilings, where one can obtain the inverse
Kasteleyn matrix from the interlaced particle system~\cite{BGR10,Pet12a}.

We initially obtained the above expression for $K^{-1}$ using a guess
based on Helfgott's formula in~\cite{Hel00} for $K^{-1}$ when $a=1$,
Theorem~\ref{localstatisticsthm} and the correlation kernel for the
particle system given in~(\ref{dimerstoparticlestwo}). In
Section~\ref{secDiscreteSetting}, we prove Theorem~\ref
{discretecouplingfunction} by verifying the equation $K \cdot K^{-1} =
\mathbb{I}$ for our conjectured formula for $K^{-1}$. In the proof, we
expand $K \cdot K^{-1}$ entrywise, which gives a five-term relation
involving entries of $K^{-1}$ due to the sparseness of the matrix $K$.
A similar set of relations (without a boundary condition) was used in
\cite{BS10}.

We can write the particle correlation kernel given in~(\ref
{dimerstoparticlestwo}) in terms of the inverse Kasteleyn matrix. The
relation is similar to that found in lozenge tiling; see~\cite
{Pet12a} and~\cite{BGR10}. Note that for lozenge tilings, the
particle correlation kernel and the kernel from the inverse Kasteleyn
matrix are in bijection.
We find the following proposition.

\begin{prop}\label{propdimerstoparticles}
%
\begin{eqnarray}\label{propeqndimerstoparticles}
\qquad && K^{-1}\bigl((x_1,x_2),(y_1,y_2)
\bigr)
\nonumber\\[-8pt]\\[-8pt]
&&\qquad  = -(-1)^{(x_1-x_2+y_1-y_2)/4} K_n \biggl( y_2,
\frac{y_2-y_1+1}{2};  x_2, \frac{x_2-x_1+1}{2} \biggr).\nonumber
\end{eqnarray}
\end{prop}

We prove this proposition in Section~\ref{secDiscreteSetting}.

\subsection{Edge fluctuations of southern dominoes}\label{sec2.2}
\label{subsecEdgefluctuationsofsoutherndominoes}

\subsubsection{Southern domino process}\label{sec2.2.1}

From our expression for the inverse Kasteleyn matrix in Theorem~\ref
{discretecouplingfunction} and using Theorem~\ref
{localstatisticsthm}, it is now possible to compute any joint
probability of dominoes.
We choose to compute the probability distribution of southern dominoes
(or equivalently dimers) in various locations of the Aztec diamond.

The southern domino process is defined as follows: fix $r$, $1 \leq r
\leq n$. With a southern domino on the line $y=r$, we mean a dimer with
a white vertex $w=(2s-1,2r)$ and a black vertex $b=(2s,2r+1)$ for some
$s \in\{1, \ldots, n\}$, and we say that the southern domino is
located at $s$ on the line $y=r$. The southern dominoes form a
determinantal process by Theorem \ref{localstatisticsthm}, and in
particular so do the southern dominoes on the line $y=r$ with the
kernel given by the following lemma.

%
\begin{lemma}
\label{lemmakernelofsoutherndominoprocess}
A kernel of the determinantal process given by the positions of the
southbound dominoes on a fixed line $y=r$ in a random tiling of an
Aztec diamond is
%
\begin{eqnarray}
\label{asymptoticskernel}
\qquad && L(x_1,x_2)
\nonumber\\[-8pt]\\[-10pt]
&&\qquad:=-\frac{ 1}{(2\pi i )^2}
\int
_{\mathcal{E}_1}dz \int_{\mathcal{E}_2} dw\,\frac{w^{x_2}}{z^{x_1}}
\frac{ (a+z)^r (a
z-1)^{n-r}}{ (aw-1)^{n-r} (a+w)^{r+1} (w-z)}.\nonumber
\end{eqnarray}
\end{lemma}

\begin{pf}
From Theorem~\ref{localstatisticsthm}, the kernel of the southern
dominoes along a fixed line $y=r$ is given by (up to conjugation)
%
\begin{equation}
K(b,\tilde{w}) K^{-1}(w,b),
\end{equation}
where we set $w=(2x_1-1,2r)$, $\tilde{w}=(2x_2-1,2r)$ and
$b=(2x_2,2r+1)$. From Theorem~\ref{discretecouplingfunction}, we have
a formula for each entry of $K^{-1}$ and in the case $x_1<x_2$, we have that
%
\begin{equation}
f_2(w,b)= \frac{(-1)^{((x_1+x_2)/2)+r}}{2\pi i} \int_{|z|={\epsilon}}
\frac{(1/a+z)^{x_2-x_1-1}}{1/a+a+z} \,dz =0
\end{equation}
because the integrand is analytic at $z=0$. From the above equation, we have
%
\begin{eqnarray}
&& K(b,\tilde{w}) K^{-1}(w,b)\nonumber
\\
&&\qquad = (-1)^{ (2 x_2+2r)/2} \frac{
(-1)^{((x_1+x_2)/2)+r}}{(2\pi i )^2 }
\\
&&\quad\qquad{}\times \int_{\mathcal{E}_1} \int_{\mathcal{E}_2} \frac{w^{x_2}}{z^{x_1}}
\frac{ (a+z)^r (a
z-1)^{n-r}}{ (aw-1)^{n-r} (a+w)^{r+1} (w-z)} \,dw \,dz.\nonumber
\end{eqnarray}
In the\vspace*{2pt} above equation, the sign is equal to
$(-1)^{(x_1+3x_2)/2}=-(-1)^{(x_2-x_1)/2}$. We can remove a factor of
$(-1)^{(x_2-x_1)/2}$ from the above equation by a conjugation which
gives~(\ref{asymptoticskernel}).
\end{pf}

\subsubsection{Thickening and thinning determinantal point processes}\label{sec2.2.2}

Consider a determinantal point process $\nu$ on a space $\Lambda$
with correlation kernel $\mathbb{K}$. Let $0\le\alpha\le1$ and
consider the
point process obtained by removing each point in the process
independently with probability $1-\alpha$. We will call this new point process
the \emph{thinned determinantal point process} with correlation kernel
$\mathbb{K}$ and parameter $\alpha$. A way of modifying the original
point process to obtain a point process with multiple points is by
taking each point in the process $\nu$ independently with multiplicity $m$,
where $m$ is a geometric random variable with parameter $\beta$,
$\mathbb{P}[m=k]=(1-\beta)\beta^{k-1}$, $k\ge1$. We will call this
(multiple) point process a \emph{thickened determinantal point
process} with correlation kernel $\mathbb{K}$
and parameter $\beta$.

%
\begin{prop} \label{thinnedthickprop}
The thinned determinantal point process $\{x_j\}$ with correlation
kernel $\mathbb{K}$ and parameter $\alpha$ is again a determinantal
point process with
correlation kernel $\alpha\mathbb{K}$, that is,
%
\begin{equation}
\label{thinnedexpectation} \mathbb{E} \bigl[e^{-\sum_{j}\psi(x_j)} \bigr
]=\det(I-\phi\alpha
\mathbb{K}\mathbb{I}_A),
\end{equation}
for every $\psi\in C_c^+(\Lambda)$, $\phi=1-e^{-\psi}$ and
$A=\operatorname{supp}\psi=\operatorname{supp}\phi$. With the same
notation the thickened determinantal point process
with kernel $\mathbb{K}$ and parameter $\beta$ is characterized by
%
\begin{equation}
\label{thickenedexpectation} \mathbb{E} \bigl[e^{-\sum_{j}\psi(x_j)}
\bigr]=\det\biggl(I-
\frac
{\phi}{1-\beta+\beta\phi}\mathbb{K}\mathbb{I}_A \biggr),
\end{equation}
where $\{x_j\}$ is now the multi-set of points of the process.
\end{prop}

We will prove the proposition in Section~\ref{sec4}. Note that the thickened process is no longer a determinantal
point process since determinantal point processes are
always simple point processes.

\subsubsection{Asymptotic coordinates}\label{sec2.2.3}
Here, we introduce the asymptotic coordinates of the unfrozen region.
As $n \rightarrow\infty$, for the (rescaled) Aztec diamond with
corners $(0,0), (1,0), (0,1)$ and $(1,1)$, the boundary between the
frozen and unfrozen regions is an ellipse~\cite{CEP96} whose equation
is given by
%
\begin{equation}
\label{ellipse} \frac{(v-u)^2}{1-p} + \frac{(u+v-1)^2}{p} =1,
\end{equation}
where $u \in[0,1]$ is the horizontal coordinate, $v\in[0,1]$ is the
vertical coordinate and $p=1/(1+a^2)$. We let $\mathcal{D}\subset
{\mathbb{R}}^2$ be the area bounded by the ellipse given in~(\ref{ellipse}).

For our results on the edge, we are interested in the boundary of the
ellipse. This is given by
%
\begin{equation}
v=1-u \pm2\sqrt{ (1-p)p(1-u)u} + p(2u-1).
\end{equation}
The arctic ellipse can be parametrized by $(u(k),v(k))$, where
$v(k)=1-k^2u(k)$ and
%
\begin{equation}
\label{asymptoticsuk} u(k)=\frac{1}{(1+a^2)(1+k^2)+2 a\sqrt{1+a^2}k}.
\end{equation}
We will be interested in two parts of the boundary.
The part where $k>0$ will be called the \emph{northern boundary}, and
the part where $k\in(-a^{-1}(1+a^2)^{1/2},-a(1+a^2)^{-1/2})=(-1/\sqrt
{1-p},-\sqrt{1-p})$ the
\emph{southern boundary}.
See Figure~\ref{fignortherboundarypicture} for an example of the
northern and southern boundaries and an explanation of the geometric
meaning of $k$.

%
\begin{figure}

\includegraphics{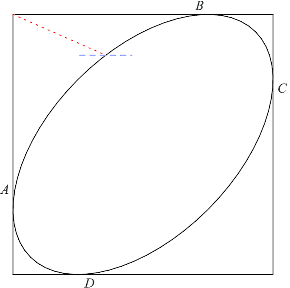}

\caption{This figure shows the intersection of the limiting ellipse
with the red dotted line given by $v=1-k^2u$. The southern dominoes on
the blue dashed line represent the southern domino process. The
northern boundary lies between $A$ and $B$ and the southern boundary
between $C$ and $D$.}
\label{fignortherboundarypicture}
\end{figure}

\subsubsection{Results on the southern domino process}\label{sec2.2.4}

Here, we consider the behavior of the southern domino process along the
northern and southern boundaries. Along the southern boundary the
southern domino
process is almost dense, and we have to consider the dual process, the
process of holes, instead. If we think of the locations in ${1,\ldots,n}$
of the southern dominoes on the line $y=r$ as positions of particles,
then the empty spaces, the holes, also form a determinantal point process
with a kernel $I-L$, where $L$ is the kernel for the particles.

In our scaling limits we will obtain the \emph{Airy kernel point
process} which is a determinantal point process with kernel
%
\begin{equation}
\label{Airykernel} K_{\Ai}(x,y)=\int_0^{\infty}
\oAi(x+t) \oAi(y+t) \,dt.
\end{equation}

Set
%
\begin{eqnarray}\label{beta}
\beta &=&-a\bigl(a+k\sqrt{1+a^2}\bigr),
\\
\label{alpha} \alpha&=&\frac{1}{1-\beta}=\frac{1}{1+a^2+ak\sqrt{1+a^2}},
\end{eqnarray}
and let $\lambda>0$ be given by
%
\begin{equation}
\label{lambda} \qquad \lambda^3=\pm\frac{a(a+k\sqrt
{1+a^2})^2}{(1+a^2)(ak^2+k\sqrt
{1+a^2})((1+a^2)(1+k^2)+2ak\sqrt{1+a^2})}
\end{equation}
with the plus sign for $k>0$ and the minus sign for $k<0$.

%
\begin{thma} \label{asymptoticsthmthinnedairy}
Let $a>0$ be fixed, and let $\lambda$ be given by (\ref{lambda}),
$\alpha$ by (\ref{alpha}) and $\beta$ by (\ref{beta}). Furthermore,
let $\{x_j\}$ be the positions of the
southern dominoes on the line $y=[(1-k^2u(k)n]$, where $u(k)$ is given
by (\ref{asymptoticsuk}). Consider a fixed $k>0$. Then the rescaled
southern domino
process at the northern boundary,
\[
\xi_j=\frac{u(k)n-x_j}{\lambda n^{1/3}},
\]
converges weakly to the thinned Airy kernel point process with
parameter $\alpha$.

Next, consider a fixed $k\in(-a^{-1}(1+a^2)^{1/2},-a(1+a^2)^{-1/2})$.
Let $\{y_j\}$ be the positions of the holes in the south domino
process, that is, the dual
south domino process, at the southern boundary. The rescaled point process
\[
\xi_j=\frac{y_j-u(k)n}{\lambda n^{1/3}}
\]
converges weakly to the thickened Airy kernel point process with
parameter~$\beta$.
\end{thma}

Successive independent thinning and rescaling of a point process
typically has a Poisson point process as its limit. If $a$ tends to
infinity, we see that the thinning parameter of the thinned Airy kernel
$\alpha$ tends to zero. Hence we can expect that if we let $a$ tend to
infinity with $n$ (but not too fast),
the southern domino point process close to the northern boundary should
converge to a Poisson process. This leads to the next theorem.

%
\begin{thma}\label{airytopoisson}
Fix $k>0$, and let $a=a(n)$, where $a(n)\to\infty$ but
$a(n)/\break  n^{1/10}\to0$ as $n\to\infty$. Set $c(a)=\pi
^{2/3}(1+1/k)^{1/3}a^{2/3}$, and let
$\{x_j\}$ be the positions of the south dominoes on the line
$y=[n(1-k^2u(k))]$. Then the rescaled point process
%
\begin{equation}
\xi_j=\frac{u(k)n+c(a)n^{1/3}-x_j}{c(a)n^{1/3}}
\end{equation}
converges weakly to a Poisson process with density $\rho(\xi)=\sqrt
{(1-\xi)_+}$.
\end{thma}

The condition on the allowed growth of $a(n)$ is certainly not optimal
but is an outcome of the proof.
A similar result holds for the thinned Airy kernel point process on
$\mathbb{R}$: if the thinning parameter is sent to zero, the Airy
kernel can be rescaled to a Poisson point process on $\mathbb{R}$
which has a square root drop off. This result actually follows from the
proof of Theorem \ref{airytopoisson}.
Thus we can think of the thinned Airy kernel point process as being an
intermediate kernel between the Airy kernel and the Poisson point
processes with density $\sqrt{(1-\xi)_+}$.

%
\begin{remark}
We could also consider the behavior of the leftmost southern
domino along the northern boundary. For a fixed $a$ we should get
convergence to the last particle
distribution for the thinned Airy kernel point process, $\det(I-\alpha
K_{\Ai})_{L^2(\xi,\infty)}$. When $a$ goes to infinity with $n$, but
not too quickly, we
expect instead get one of the classical extreme value distributions in
the limit, namely the last particle distribution in
a Poisson process with density $\sqrt{(1-\xi)_+}$. We will not give
the technical details that are required to prove these natural
conjectures, but it should be
possible by developing the proof of Theorem \ref{airytopoisson} further.
\end{remark}

\subsection{Bulk fluctuations}\label{sec2.3}\label{subsecBulk}

An account of \emph{local Gibbs measures} for tiling models can be
found in \cite{KOS06} and \cite{Ken09}.

For all doubly periodic bipartite weighted dimer models embedded in the
plane, in \cite{KOS06} the authors found that the dimer model is a
Gibbs measure, gave an explicit method to compute the entries of the
inverse Kasteleyn matrix embedded in the plane and the complete phase
portrait. The results from \cite{KOS06} rely on using the smallest
nonrepeating unit of the graph called the \emph{fundamental domain}.
For the graph considered in this paper, the fundamental domain has one
black vertex and one white vertex. In order to describe the Gibbs
measure, the authors of \cite{KOS06} introduced magnetic coordinates
$(B_x,B_y)$, where one increases the energy by $e^{B_x}$ or $e^{B_y}$
if one passes to the neighboring fundamental domain to the left or
above. Conversely, if one passes to the fundamental domain to the right
or below, one decreases the energy by $e^{B_x}$ or $e^{B_y}$. These
magnetic coordinates are related
to the average slope; that is, one can compute the Gibbs measures for
different slopes; see \cite{KOS06}.

We choose the fundamental domain of the graph embedded in the plane to
be given by a white vertex, an edge in the direction $+e_2$ and its
incident black vertex and the remaining edges incident to these vertices.
To make the following computations and formulas simpler and since the
dimer model is independent of the chosen Kasteleyn orientation, we
choose the Kasteleyn orientation which multiplies the Kasteleyn
orientation given in Section~\ref{subsecTheKasteleynmatrix} by
$(-1)$ at the black vertices $(b_1,b_2)$ where $b_1+b_2\operatorname
{mod}4=3$. Figure~\ref{figfundamentaldomain} shows our choice of
fundamental domain.
%
\begin{figure}

\includegraphics{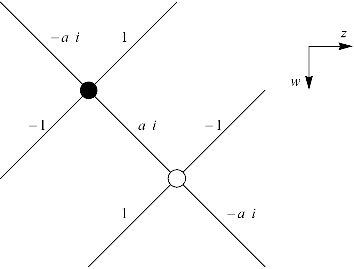}

\caption{The fundamental domain.}\label{figfundamentaldomain}
\end{figure}

We denote the Gibbs measure of the model on this graph by $\mu
_a(B_x,B_y)$ where $(B_x,B_y)$ is described above. Suppose that
$(2\alpha_1+1,2\alpha_2)$ is a white vertex, and $(2\beta_1,2 \beta
_2+1)$ is a black vertex for $\alpha_1, \alpha_2,\beta_1,\beta_2
\in\mathbb{Z}$. Using techniques from \cite{KOS06} one can find the
entries of the inverse of the (infinite) Kasteleyn matrix, denoted by
$K^{-1}_\mu$, and they are given by
%
\begin{eqnarray}\label{Gibbskast}
&& K_\mu^{-1}\bigl((2\alpha_1+1,2
\alpha_2), (2\beta_1,2 \beta_2+1)\bigr)
\nonumber\\[-8pt]\\[-8pt]
&&\qquad  =
\frac
{1}{(2\pi i)^2} \int_{|z|=1} \int_{|w|=1}
\frac{ z^{\alpha_1-\beta
_1} w^{\beta_2-\alpha_2}} { P(z e^{B_x}, w e^{B_y}) } \frac{dw}{w} \frac{dz}{z},\nonumber
\end{eqnarray}
where $P(z,w)$ is the so-called \emph{characteristic polynomial}. For
the above graph embedded in the torus with the above edge weights and
Kasteleyn orientation, the characteristic polynomial is given by
%
\begin{equation}
P(z,w)= a i -z^{-1} +w^{-1} - a i w^{-1}
z^{-1}.
\end{equation}


%
\begin{thma}\label{localGibbsmeasure}
Choose the rescaling so that the white vertices are given by
\[
(x_1,x_2)=\bigl([2\xi_1 n]+ 2
\alpha_1 +1, [2\xi_2 n] +2 \alpha_2\bigr)
\]
and the black vertices are given by
\[
(y_1,y_2)=\bigl([2\xi_1 n] + 2
\beta_1,  [2\xi_2 n] +2 \beta_2+1\bigr)
\]
for $\xi_1, \xi_2 \in\mathcal{D}_c \subset\mathcal{D} $ compact
and for $\alpha_1,\alpha_2, \beta_1,\beta_2 \in\mathbb{Z}$ where
$\mathcal{D}$ is the area bounded by the ellipse in~(\ref{ellipse}).
Then the measure on domino tilings converges weakly to $\mu_a(\log
r_1, \log r_2)$ where $\mu_a$ is defined above and
%
\begin{equation}
\label{bulkdefr} r_i =\sqrt{\xi_i/(1-
\xi_i)}\qquad \mbox{for } i\in\{1,2\}.
\end{equation}
\end{thma}




Similar results for certain classes of lozenge tilings have been
obtained in~\cite{BGR10} and~\cite{Pet12a}.

\subsection{Overview of the paper}\label{sec2.4}\label{subsecOverviewofthePaper}
The rest of the paper is organized as follows. In Section~\ref
{secDiscreteSetting}, we prove Theorem~\ref
{discretecouplingfunction} and Proposition~\ref
{propdimerstoparticles}. In Section~\ref{sec4},
we give the~proofs of Theorems~\ref{asymptoticsthmthinnedairy} and~\ref{airytopoisson}. The proof of Theorem~\ref
{localGibbsmeasure} is given in Section~\ref{secBulk}. We conclude
with a brief discussion of the height function fluctuation in
Section~\ref{conclusion}.\

\section{Discrete setting}\label{sec3}\label{secDiscreteSetting}

\subsection{Proof of Theorem \texorpdfstring{\protect\ref{discretecouplingfunction}}{2.3}}\label{sec3.1}\label{subsecProofofdiscretecouplingfunctiontheorem}

Before giving the proof of Theorem~\ref{discretecouplingfunction}, we
introduce some notation: for $x=(x_1,x_2) \in\mathtt{W}$ and
$y=(y_1,y_2) \in\mathtt{B}$, let
%
\begin{eqnarray}
\qquad c_1(w,z,x,y) &=& (-1)^{(y_1+y_2+x_1+x_2)/4}
\nonumber\\[-8pt]\label{discretec1} \\[-8pt]
&&{}\times
\frac{ w^{y_1/2}
(a+z)^{x_2/2} (a z -1)^{(2n-x_2)/2} }{ z^{(x_1+1)/2} (w-z) (a
w-1)^{(2n+1-y_2)/2} (a+w)^{(y_2+1)/2}},\nonumber
\\
c_2(z,x,y)&=& (-1)^{(y_1+y_2+x_1+x_2)/4}a^{(y_2-x_2-1)/2}
\nonumber\\[-8pt]\label{discretec2} \\[-8pt]
&&{}\times \frac{
z^{(y_2-x_2-1)/2} (1/a +z)^{(y_1-x_1-1)/2}}{ (1/a+a+z)^{(y_2-x_2+1)/2}}\nonumber
\end{eqnarray}
and
%
\begin{eqnarray}\label{discretect2}
\tilde{c}_2(w,x,y) &=& (-1)^{(y_1+y_2+x_1+x_2)/4}
w^{(y_1-x_1-1)/2}
\nonumber\\[-8pt]\\[-8pt]
&&{}\times  (a w-1)^{(y_2-x_2-1)/2} (a+w)^{(x_2-y_2-1)/2}.\nonumber
\end{eqnarray}
We have chosen the three functions above so that
%
\begin{equation}
\label{discretec1tof1} \frac{1}{(2\pi i)^2} \int_{\mathcal{E}_2} \int
_{\mathcal{E}_1} c_1(w,z,x,y) \,dz \,dw=f_1(x,y)
\end{equation}
and
%
\begin{equation}
\label{discretec2tof2} \frac{1}{2 \pi i} \int_{\mathcal{E}_1}
c_2(z,x,y) \,d z = \frac
{1}{2\pi i} \int_{\mathcal{E}_2}
\tilde{c}_2(w,x,y) \,dw =f_2 (x,y),
\end{equation}
where the contours $\mathcal{E}_1$ and $\mathcal{E}_2$ are given in
the statement of Theorem~\ref{discretecouplingfunction}, and the
functions $f_1(x,y)$ and $f_2(x,y)$ are given in equations~(\ref
{functionf1}) and~(\ref{functionf2}), respectively.
Note that $c_2$ is obtained from $\tilde{c}_2$ by the change of
variables $z\mapsto w-1/a$ and
%
\begin{equation}
\label{discretec1toc2} \lim_{z \to w} (w-z) c_1(w,z,x,y)=
\tilde{c}_2(w,x,y).
\end{equation}

\begin{pf*}{Proof of Theorem \ref{discretecouplingfunction}}
For the proof, we set $x=(x_1,x_2)$ and $y=(y_1,y_2)$ with $x,y \in
\mathtt{B}$. We keep the same notation throughout the proof.
The matrix $K^{-1}$ is uniquely determined by the specific choice of
the Kasteleyn matrix, $K$, which means we need to verify the equation
$K \cdot K^{-1} =\mathbb{I}$. We can expand out $K \cdot K^{-1}$
entry-wise. For $x,y \in\mathtt{B}$, we obtain
%
\begin{eqnarray}
\label{leminverseKC} K\cdot K^{-1} (x,y) & =&\sum_{w \in\mathtt{W}}
K (x,w) K^{-1}(w,y)
 = \sum_{w \sim x,w \in\mathtt{W}} K(x,w) K^{-1}(w,y),\hspace*{-30pt}
\end{eqnarray}
where $w \sim x, w \in\mathtt{W}$ means that $w \in\mathtt{W}$ and
$w$ is a nearest neighbored vertex to~$x$. Using the entries of the
Kasteleyn matrix given in~(\ref{kasteleynweights}), we can
rewrite~(\ref{leminverseKC}) and compare with the identity matrix
which gives an entry-wise expansion of the equation $K \cdot K^{-1} =
\mathbb{I}$. This is the equation we must verify to prove Theorem~\ref
{discretecouplingfunction}. That is, we must verify
%
\begin{eqnarray}
\label{discreteverify} &&(-1)^{(x_1+x_2-1)/2} \bigl( K^{-1}(x+e_1,y)
\mathbb{I}_{x_1<2n}-K^{-1}(x-e_1,y)
\mathbb{I}_{x_1>0}\nonumber
\\
&&\hspace*{72pt}{}  - a i K^{-1} (x+e_2,y)\mathbb{I}_{x_1>0}+a
i K^{-1}(x-e_2,y) \mathbb{I}_{x_1<2n} \bigr)
\\
&&\qquad  = \mathbb{I}_{x=y},
\nonumber
\end{eqnarray}
where $x=(x_1,x_2),y \in\mathtt{B}$ and
%
\begin{equation}
\mathbb{I}_{x_1>0} = \cases{ 1, &\quad if $x_1>0$,
\vspace*{3pt}\cr
0, &
\quad otherwise.}
\end{equation}
Note that the indicator functions in~(\ref{discreteverify}) account
for $x$ on the boundary of the Aztec diamond.
In order to verify~(\ref{discreteverify}), there are three cases to
consider for $x=(x_1,x_2)$: $0<x_1<2n$, $x_1=0$ and $x_1=2n$.

For $0<x_1<2n$, the left-hand side of~(\ref{discreteverify}) is equal to
%
\begin{eqnarray}\label{leminverseinterior}
&& (-1)^{(x_1+x_2-1)/2} \bigl( K^{-1}(x+e_1,y)
-K^{-1}(x-e_1,y)
\nonumber\\[-8pt]\\[-8pt]
&&\hspace*{72pt} {} - a i K^{-1}
(x+e_2,y)+a i K^{-1}(x-e_2,y) \bigr).\nonumber
\end{eqnarray}
We first substitute $f_1$ into the above expression. For this
expression, we will manipulate the integrand of $f_1$ which is given by
$c_1$ by~(\ref{discretec1tof1}).
We find after some simplification,
%
\begin{eqnarray}\label{discreteinteriorverify}
&& (-1)^{(x_1+x_2-1)/2} \bigl(c_1(w,z,x+e_1,y)-c_1(w,z,x-e_1,y) \nonumber
\\
&&\hspace*{72pt}{} - a i c_1
(w,z,x+e_2,y)+a i c_1(w,z,x-e_2,y) \bigr)\nonumber
\\
&&\qquad  =(-1)^{(x_1+x_2-1)/2}\nonumber
\\
&&\quad\qquad{} \times  \bigl( c_1\bigl(w,z,(x_1+1,x_2+1),y
\bigr) -c_1\bigl(w,z,(x_1-1,x_2-1),y\bigr)
\\
&&\hspace*{48pt} {} - a i c_1 \bigl(w,z,(x_1-1,x_2+1),y
\bigr)+a i c_1\bigl(w,z,(x_1+1,x_2-1),y\bigr)
\bigr)\nonumber\hspace*{-10pt}
\\
&&\qquad =(-1)^{(x_1+x_2-1)/2}\nonumber
\\
&&\quad\qquad{}\times c_1(w,z,x+e_1,y) \biggl( 1-a z+
\frac{a (-1+a z)}{a+z}+\frac{z (-1+a z)}{a+z} \biggr)
=0.\nonumber
\end{eqnarray}
Note that this relation holds for $0\leq x_1 \leq2n$.
Integrating both sides of the above equation with respect to $z$ and
$w$ over the contours $\mathcal{E}_1$ and $\mathcal{E}_2$,
respectively, and using~(\ref{discretec1tof1}), we find that
%
\begin{eqnarray}\label{discreteinteriorf1}
&& (-1)^{(x_1+x_2-1)/2} \bigl( f_1(x+e_1,y) -f_1(x-e_1,y)
\nonumber\\[-8pt]\\[-8pt]
&&\hspace*{72pt}{}  - a i f_1 (x+e_2,y)+a i f_1(x-e_2,y) \bigr)=0.\nonumber
\end{eqnarray}
To substitute $f_2$ into the expression given in~(\ref
{leminverseinterior}), we have to consider $x_1=y_1$ and $x_1 \geq
y_1+2$ separately due to the split expression of $K^{-1}$.
For $x_2 \geq y_1+2$, all four terms of $f_2$ are present in~(\ref
{leminverseinterior}) and so using~(\ref{discretec1toc2}),~(\ref
{discreteinteriorverify}) and~(\ref{discretec2tof2}) we find
%
\begin{eqnarray}\label{discreteinteriorf2first}
&& -(-1)^{(x_1+x_2-1)/2} \bigl( f_2(x+e_1,y)
-f_2(x-e_1,y)
\nonumber\\[-8pt]\\[-8pt]
&&\hspace*{80pt}{} - a i f_2 (x+e_2,y)+a
i f_2(x-e_2,y) \bigr)=0.\nonumber
\end{eqnarray}
We now substitute $f_2$ into~(\ref{leminverseinterior}) for the case
$x_1=y_1$. We obtain
%
\begin{eqnarray}
&& -(-1)^{(x_1+x_2-1)/2} \bigl( f_2\bigl((x_1+1,x_2+1),(x_1,y_2)
\bigr)
\nonumber\\[-8pt]\\[-8pt]
&&\hspace*{81pt}{} +a i f_2\bigl((x_1+1,x_2-1),(x_1,y_2)
\bigr) \bigr).\nonumber
\end{eqnarray}
We first manipulate the integrand of the above equation using~(\ref
{discretec2tof2}) which gives
%
\begin{eqnarray}
&& - (-1)^{(x_1+x_2-1)/2} \bigl( c_2\bigl(z,(x_1+1,x_2+1),(x_1,y_2)
\bigr)\nonumber
\\
&&\hspace*{81pt}{} +a i c_2\bigl(z,(x_1+1,x_2-1),(x_1,y_2)
\bigr) \bigr)
\nonumber
\\
&&\qquad  =- (-1)^{(x_1+x_2-1)/2}c_2\bigl(z,(x_1+1,x_2+1),(x_1,y_2)
\bigr) \biggl( 1+ \frac{a^3 z}{1+a^2+a z} \biggr)\hspace*{-15pt}
\nonumber
\\[-8pt]
\\[-8pt]
&&\qquad =- (-1)^{(x_1+x_2-1)/2}c_2\bigl(z,(x_1+1,x_2+1),(x_1,y_2)
\bigr) \biggl( \frac{ (1+a^2 ) (1+a z)}{1+a^2+a z} \biggr)\nonumber\hspace*{-15pt}
\\
&&\qquad  = -(-1)^{(3x_1+3x_2+y_1+y_2)/4} a^{(y_2-x_2-2)/2 } \nonumber
\\
&&\hspace*{6pt}\quad\qquad{} \times \bigl(1+a^2\bigr)
z^{(y_2-x_2-2)/2} \biggl(\frac{1}{a} +a+z \biggr)^{(x_2-y_2-2)/2}.
\nonumber
\end{eqnarray}
We now integrate with respect to $z$ over the contour $\mathcal{E}_1$,
and we obtain
%
\begin{eqnarray}
&&-(-1)^{(x_1+x_2-1)/2} \bigl( f_2\bigl((x_1+1,x_2+1),(x_1,y_2)
\bigr)\nonumber
\\
&&\hspace*{81pt}{}  +a i f_2\bigl((x_1+1,x_2-1),(x_1,y_2)\bigr) \bigr)\nonumber
\\
&&\qquad  =-\frac{(-1)^{(3x_1+3x_2+y_1+y_2)/4}} {2\pi i }
\\
&&\hspace*{6pt}\quad\qquad{}\times  \int_{\mathcal
{E}_1}a^{(y_2-x_2-2)/2 }
\bigl(1+a^2\bigr) z^{(y_2-x_2-2)/2} \biggl(\frac{1}{a} +a+z
\biggr)^{(x_2-y_2-2)/2} dz\nonumber\hspace*{-18pt}
\\
&&\qquad  = \cases{ -(-1)^{(x_1+x_2)}, &\quad$x_2=y_2$,
\vspace*{3pt}\cr
0,
&\quad otherwise}\nonumber
\end{eqnarray}
by Lemma~\ref{lemtechnical} below. Because $x_1+x_2$ is always odd,
we conclude
%
\begin{eqnarray}
\label{discreteinteriorf2second}
&& -(-1)^{(x_1+x_2-1)/2} \bigl(
f_2\bigl((x_1+1,x_2+1),(x_1,y_2)
\bigr)
\nonumber\\[-8pt]\\[-8pt]
&&\hspace*{81pt}{} +a i f_2\bigl((x_1+1,x_2-1),(x_1,y_2)
\bigr) \bigr)=\mathbb{I}_{x_2=y_2}.\nonumber
\end{eqnarray}
Note that by our method of computation, the above relation is valid for
$0 \leq x _1 \leq2n$.
This means we have computed~(\ref{leminverseinterior}) for $0 < x_1
<2n$ and so from~(\ref{discreteinteriorf1}), (\ref
{discreteinteriorf2first})~and~(\ref{discreteinteriorf2second}), we
have obtained for $0<x_1 < 2n$,
%
\begin{eqnarray}\label{leminverseinterior1}
&& (-1)^{(x_1+x_2-1)/2} \bigl( K^{-1}(x+e_1,y)
-K^{-1}(x-e_1,y)
\nonumber\\[-8pt]\\[-8pt]
&&\hspace*{73pt}{} - a i K^{-1}
(x+e_2,y)+a i K^{-1}(x-e_2,y) \bigr) =
\mathbb{I}_{x=y}.\nonumber
\end{eqnarray}

For $x=(0,x_2)$, we have that the left-hand side of~(\ref
{discreteverify}) is equal to
%
\begin{equation}
\label{discreteleft} (-1)^{(x_2-1)/2} \bigl( K^{-1}(x+e_1,y)+a
i K^{-1}(x-e_2,y) \bigr).
\end{equation}
Before we substitute $f_1$ into~(\ref{discreteleft}), notice that
%
\begin{equation}
\label{discreteleftf1} f_1\bigl((-1,w_2),(y_1,y_2)
\bigr)=0
\end{equation}
for $w_2\, \operatorname{mod}2 =0$ because there is no residue at $z=0$
in~(\ref{functionf1}) in this case. Since~(\ref
{discreteinteriorverify}) holds for $x$ (including those outside of
the Aztec diamond) which means that the relation in~(\ref
{discreteinteriorf1}) holds for any values of $x$, we can write
out~(\ref{discreteinteriorverify}) with $x_1=0$, integrate over $z$
and $w$ over the contours $\mathcal{E}_1$ and $\mathcal{E}_2$,
respectively, noting~(\ref{discreteleftf1}), and use~(\ref
{discretec1tof1}) to obtain
%
\begin{eqnarray}
\label{discreteleftf1second} &&(-1)^{(x_2-1)/2} \bigl( f_1\bigl
((1,x_2+1),y
\bigr)-f_1\bigl((-1,x_2-1),y\bigr)\nonumber
\\
&&\hspace*{58pt}
{} -a i f_1
\bigl((-1,x_2+1),y\bigr)+ a i f_1\bigl((1,x_2-1),y
\bigr) \bigr)
\\
&&\qquad  =(-1)^{(x_2-1)/2} \bigl( f_1(x+e_1,y)+a i
f_1(x-e_2,y) \bigr)=0
\nonumber
\end{eqnarray}
for $x=(0,x_2)$.
When we substitute $f_2$ into~(\ref{discreteleft}), we only need to
consider the case $y_1=0$ because of the split definition of $K^{-1}$,
and so using~(\ref{discreteinteriorf2second}) (because the equation
is valid for $0 \leq x_1 \leq2n$), we obtain
%
\begin{eqnarray}
\label{discreteleftf2}
&& -(-1)^{(x_2-1)/2} \bigl( f_2\bigl
((1,x_2+1),(0,y_2)
\bigr)
\nonumber\\[-8pt]\\[-8pt]
&&\hspace*{67pt} {}+a i f_2\bigl((1,x_2-1),(0,y_2)\bigr)
\bigr)=\mathbb{I}_{x_2=y_2}.\nonumber
\end{eqnarray}
Adding~(\ref{discreteleftf1second}) and~(\ref{discreteleftf2}), we find
%
\begin{equation}
\label{discreteleft1} (-1)^{(x_2-1)/2} \bigl( K^{-1}(x+e_1,y)+a
i K^{-1}(x-e_2,y) \bigr)=\mathbb{I}_{x=y}
\end{equation}
for $x=(0,x_2)$.

For $x=(2n,x_2)$, we have that the left-hand side of~(\ref
{discreteverify}) is equal to
%
\begin{equation}
\label{discreteright} (-1)^{(2n+x_2-1)/2} \bigl(- K^{-1} (
x-e_1,y) -a i K^{-1}(x+e_2,y) \bigr).
\end{equation}
For $x=(2n,x_2)$, we perform the following computation:
%
\begin{eqnarray}
\label{discreterightf1tof2} &&(-1)^{(2n+x_2-1)/2} \bigl( -
f_1\bigl((2n-1,x_2-1),y
\bigr) -a i f_1\bigl((2n-1,x_2+1),y\bigr) \bigr)
\nonumber
\\
&&\qquad =-\frac{(-1)^{(2n+x_2-1)/2}}{(2\pi i)^2}\int_{\mathcal{E}_2} \int
_{\mathcal{E}_1}
c_1\bigl(w,z,(2n-1,x_2-1),y\bigr)\nonumber
\\
&&\hspace*{159pt}{} + a i c_1 \bigl(w,z,(2n-1,x_2+1),y\bigr) \,dz \,dw
\\
&&\qquad  =-\frac{(-1)^{(2n+x_2-1)/2}}{(2\pi i)^2}\nonumber
\\
&&\hspace*{6pt}\quad\qquad{}\times \int_{\mathcal{E}_2} \int
_{\mathcal{E}_1}c_1\bigl(w,z,(2n-1,x_2-1),y\bigr) \biggl(1-
\frac{a (a+z)}{a z-1} \biggr) \,dz \,dw\nonumber
\\
&&\qquad  = -\frac{(-1)^{(2n+x_2-1)/2}}{(2\pi i)^2}\nonumber
\\
&&\hspace*{6pt}\quad\qquad{}\times  \int_{\mathcal
{E}_2} \int_{\mathcal{E}_1}
c_1\bigl(w,z,(2n-1,x_2-1),y\bigr) \biggl(
\frac{1+ a^2}{ a z-1} \biggr) \,dz \,dw\nonumber
\\
&&\qquad  =- \frac{(-1)^{(2n+x_2-1)/2}}{(2\pi i)^2}\int_{\mathcal
{E}_2} \tilde{c}_2
\bigl(w,(2n-1,x_2-1),y\bigr) \frac{1+a^2}{1-a w} \,dw
\nonumber
\\
&&\qquad  =-(-1)^{(2n+x_2-1)/2} \bigl( f_2\bigl((2n-1,x_2-1),y
\bigr) +a i f_2\bigl((2n-1,x_2+1),y\bigr) \bigr),\nonumber\hspace*{-8pt}
\end{eqnarray}
where the fourth line to the fifth line follows from the fact that the
integrand in the fourth line is a polynomial of degree $n-1$ in the
numerator and a polynomial of degree $n+1$ in the denominator with
respect to $z$, and so we can push the contour through infinity which
picks up a residue at $z=w$. The sixth line follows from the fifth line because
%
\begin{eqnarray}
&&\tilde{c}_2\bigl(w,(2n-1,x_2-1),y\bigr) +a i \tilde{c}_2\bigl(w,(2n-1,x_2+1),y\bigr)
\nonumber
\\
&&\qquad  =\tilde{c}_2\bigl(w,(2n-1,x_2-1),y\bigr) \biggl(1-
\frac{a (a+w)}{-1+a w} \biggr)
\\
&&\qquad  = \tilde{c}_2\bigl(w,(2n-1,x_2-1),y\bigr)
\frac{1+a^2}{1-a w}
\nonumber
\end{eqnarray}
and integrating over $\mathcal{E}_2$ using~(\ref{discretec2tof2}).
For $x_1=2n$ and $y_1<2n$, we have that~(\ref{discreteright}) is
equal to
%
\begin{eqnarray}
\label{discreteright2}
&& -(-1)^{(2n+x_2-1)/2}\bigl(
f_1(x-e_1,y)-f_2(x-e_1,y)
\nonumber\\[-8pt]\\[-8pt]
&&\hspace*{82pt}
{}+ai \bigl(f_1(x+e_2,y)-f_2(x+e_2,y)
\bigr)\bigr)=0\nonumber
\end{eqnarray}
by~(\ref{discreterightf1tof2}). For $x=(2n,x_2)$ and $y_1=2n$,
using~(\ref{discreterightf1tof2}),~(\ref{discreteright}) is equal to
%
\begin{eqnarray}
\label{discreteright3}
&& -(-1)^{(2n+x_2-1)/2}\bigl(f_1(x-e_1,y)+aif_1(x+e_2,y)
\bigr)\nonumber
\\
&&\qquad  = -(-1)^{(2n+x_2-1)/2}\bigl( f_2(x-e_1,y)+aif_2(x+e_2,y)
\bigr)
\nonumber\\[-8pt]\\[-8pt]
&&\qquad  =(-1)^{(2n+x_2-1)/2}\bigl( f_2(x+e_1,y)+aif_2(x-e_2,y)
\bigr)\nonumber
\\
&&\qquad  =\mathbb{I}_{x=y}\nonumber
\end{eqnarray}
for $y=(2n,y_2)$ by using~(\ref{discreteinteriorf2first}) and~(\ref
{discreteinteriorf2second}). From~(\ref{discreteright2}) and~(\ref
{discreteright3}), we have evaluated~(\ref{discreteright}) and have found
%
\begin{equation}
(-1)^{(2n+x_2-1)/2} \bigl(- K^{-1} ( x-e_1,y) -a i
K^{-1}(x+e_2,y) \bigr)=\mathbb{I}_{x=y}
\label{discreterightfinal}
\end{equation}
for $x=(2n,x_2)$ and $y=(y_1,y_2)$.
Equations~(\ref{leminverseinterior1}),~(\ref{discreteleft1})
and~(\ref{discreterightfinal}) means that we have verified~(\ref
{discreteverify}).
\end{pf*}

%
\begin{lemma} \label{lemtechnical}
For $k \in{\mathbb{Z}}$,
%
\begin{eqnarray}
\label{lemintegralstatement} \qquad\frac{1}{2\pi i} \int_{|z|=1}
z^{k-1} \biggl( \frac{1}{a} +a + z \biggr)^{-1-k} \,dz =
\cases{ \displaystyle\frac{a}{1+a^2}, &\quad$k=0$,
\vspace*{3pt}\cr
0, &\quad
otherwise.}
\end{eqnarray}
\end{lemma}

\begin{pf}
For $k=0$, the left-hand side of (\ref{lemintegralstatement}) is
equal to
%
\begin{equation}
\frac{1}{2\pi i }\int_{|z|=1} \frac{1}{z((1/a)+a+z)}\,dz =
\frac{a}{1+a^2}.
\end{equation}
When $k>0$, the integrand in (\ref{lemintegralstatement}) is
analytic at $z=0$ and so the left-hand side of (\ref
{lemintegralstatement}) is zero. When $k<0$, we can move the contour
to a small circle around $-(a+1/a)$ and use the fact that the integrand
is analytic inside.
\end{pf}

\subsection{Guessing $K^{-1}$}\label{sec3.2}

As mentioned above in~\cite{Joh05}, the author used a particle system
formed from the zig-zag particles and obtained a formula for the
correlation kernel. From this correlation kernel, we could guess an
expression for the inverse Kasteleyn matrix which is verified to be
correct in the previous section. Here, we describe the steps we used to
obtain the guess.

Let $w \in\mathtt{W}$ and $b \in\mathtt{B}$. Recall that there is a
blue particle at $w$ if and only if a dimer covers $(w+e_1,w)$ or
$(w+e_2,w)$ and that there is a red particle at $b$ if and only if a
dimer covers $(b,b-e_1)$ or $(b,b-e_2)$. From~(\ref{kasteleynweights})
we see that if $w=(x_1,x_2)$, then
%
\begin{eqnarray}\label{dimerstoparticles4}
K(w+e_1,w)&=&(-1)^{(x_1+x_2-1)/2},
\nonumber\\[-8pt]\\[-8pt]
K(w+e_2,w)&=&(-1)^{(x_1+x_2-1)/2} a i\nonumber
\end{eqnarray}
and if $b=(y_1,y_2)$, then
%
\begin{eqnarray}\label{dimerstoparticles5}
K(b,b-e_1)&=&(-1)^{(y_1+y_2+1)/2},
\nonumber\\[-8pt]\\[-8pt]
K(b,b-e_2)&=&-(-1)^{(y_1+y_2+1)/2} a i.\nonumber
\end{eqnarray}
It follows from Theorem~\ref{localstatisticsthm} that
%
\begin{eqnarray}
\label{dimerstoparticles6}
&&\mathbb{P}[\mbox{There are particles at } w \mbox{ and } b]\nonumber
\\
&&\qquad  = \sum_{r_1,r_2=1}^2 K (w+e_{r_1},w
) K ( b,b-e_{r_2} )\nonumber
\\
&&\hspace*{60pt}{}\times  \left|\matrix{K^{-1}(w,w+e_{r_1})
& K^{-1}(w,b)
\vspace*{3pt}\cr
K^{-1}(b-e_{r_2},
w+e_{r_1}) & K^{-1}(b-e_{r_2},b) }\right|
\\
&&\qquad  =\left|\matrix{
\displaystyle \sum_{r=1}^2 K^{-1}(w,w+e_r) K(w+e_r,w)
\vspace*{3pt}\cr
\displaystyle \sum_{r_1,r_2=1}^2 K^{-1}(b-e_{r_2},w+e_{r_1})K(w+e_{r_1},w) K(b,b-e_{r_2} )} 
\right.\nonumber
\\[3pt]
&&\hspace*{148pt} \left. \matrix{
K^{-1}(w,b)
\vspace*{3pt}\cr
\displaystyle\sum_{r=1}^2 K^{-1} (b-e_r,b)K(b,b-e_r) } \right|.\nonumber
\end{eqnarray}
In~(\ref{dimerstoparticles6}), we have $w=(x_1,x_2)$, $b=(y_1,y_2)$
and if $(v_1,v_2)$, the particle coordinates, are related to
$(x_1,x_2)$ by~(\ref{dimerstoparticles1}) and $(v_1,v_2)$ in the same
way to $(y_1,y_2)$, then we get, using the particle kernel~(\ref
{dimerstoparticlestwo}) and the result in~\cite{Joh05} that
%
\begin{eqnarray}\label{dimerstoparticles7}
&& \mathbb{P}[\mbox{There are particles at } w
\mbox{ and } b]
\nonumber\\[6pt]\\[-18pt]
&&\qquad =
\left|\matrix{ K_n(u_1,u_2;u_1,u_2)
& K_n(u_1,u_2;v_1,v_2)
\vspace*{5pt}\cr
K_n(v_1,v_2l u_1,u_2)
& K_n(v_1,v_2l v_1,
v_2) }\right|.\nonumber
\end{eqnarray}
Comparing~(\ref{dimerstoparticles6}) and~(\ref
{dimerstoparticles7}), we see that it is reasonable to expect that
%
\begin{eqnarray}
&& K^{-1}\bigl((x_1,x_2),(y_1,y_2)
\bigr)
\nonumber\\[-8pt]\\[-8pt]
&&\qquad  = c(x_1,x_2;y_1,y_2)
K_n(u_1,u_2;v_1,v_2)\nonumber
\end{eqnarray}
or
%
\begin{eqnarray}\label{dimerstoparticles9}
&& K^{-1}\bigl((x_1,x_2),(y_1,y_2)
\bigr)
\nonumber\\[-8pt]\\[-8pt]
&&\qquad  = c(x_1,x_2;y_1,y_2)
K_n(v_1,v_2;u_1,u_2)\nonumber
\end{eqnarray}
with some appropriate, hopefully simple function $c$ which could
perhaps be just a sign factor. Here, one has to make some guesses and
it turns out, a posteriori, that~(\ref{dimerstoparticles9}) is the
right choice and that
%
\begin{equation}
c(x_1,x_2;y_1,y_2)=-(-1)^{(x_1-x_2+y_1-y_2+2)/4}
\end{equation}
for our choice of Kasteleyn orientation. Thus we write
%
\begin{eqnarray}
\qquad && K^{-1} \bigl((x_1,x_2),(y_1,y_2)
\bigr)
\nonumber\\[-8pt]\\[-8pt]
&&\qquad =-(-1)^{(x_1-x_2+y_1-y_2+2)/4} K_n \biggl( y_2,
\frac{y_2-y_1+1}{2};  x_2, \frac{x_2-x_1+1}{2} \biggr).\nonumber
\end{eqnarray}

\subsection{Proof of Proposition~\texorpdfstring{\protect\ref{propdimerstoparticles}}{2.4}}\label{sec3.3}
We will show that the right-hand side of~(\ref
{propeqndimerstoparticles}) gives the corresponding entry of the
inverse Kasteleyn matrix.

Write $y_2=2r-1$ and $x_2=2s$. From~(\ref{dimerstoparticlesKtilde}),
we see that
%
\begin{eqnarray}
&&\widetilde{K}_n \biggl( y_2,  \frac{y_2-y_1+1}{2};
x_2, \frac{x_2-x_1+1}{2} \biggr)\nonumber
\\
&&\qquad  =\frac{1}{(2\pi i)^2}\nonumber
\\
&&\quad\qquad{}\times  \int_{\gamma_{r_1}} \frac{dw}{w} \int
_{\gamma_{r_2}} \frac{dz}{z} \frac{ z^{(x_2-x_1+1)/2} (1-a
z)^{n-(x_2/2)} (1+a/z)^{(x_2)/2}} { w^{(y_2-y_1+1)/2} (1-a w)^{n-((y_2+1)/2) +1} (1+a/w)^{(y_2+1)/2}}\hspace*{-30pt}\nonumber
\\
&&\hspace*{113pt}{}\times  \frac{w}{w-z}\nonumber
\\
&&\qquad  =\frac{(-1)^{(y_2-x_2-1)/2}}{(2\pi i )^2}
\nonumber\\[-8pt]\\[-8pt]
&&\quad\qquad{}\times  \int_{\gamma
_{r_1}} \frac{dw}{w} \int
_{\gamma_{r_2}} \frac{dz}{z} \frac{
w^{(y_1)/2} (az-1)^{(2n-x_2)/2} (z+a)^{(x_2)/2}} {
z^{(x_1+1)/2} (a w-1)^{((2n-y_2+1)/2) +1} (w+a)^{(y_2+1)/2}}\nonumber
\\
&&\hspace*{113pt}{}\times  \frac{1}{w-z}\nonumber
\\
&&\qquad  =\frac{(-1)^{(y_2-x_2-1)/2}}{(2\pi i )^2}\nonumber
\\
&&\quad\qquad{} \times  \int_{\mathcal
{E}_2} \frac{dw}{w} \int
_{\mathcal{E}_1} \frac{dz}{z} \frac{
w^{(y_1)/2} (az-1)^{(2n-x_2)/2} (z+a)^{(x_2)/2}} {
z^{(x_1+1)/2} (a w-1)^{((2n-y_2+1)/2) +1} (w+a)^{(y_2+1)/2}}\nonumber
\\
&&\hspace*{106pt}{}\times  \frac{1}{w-z},\nonumber
\end{eqnarray}
where the last equality follows by deforming $\gamma_{r_2}$ to
$\mathcal{E}_2$ through infinity. Hence, we obtain
%
\begin{eqnarray}
&& -(-1)^{(x_1-x_2+y_1-y_2+2)/4} \widetilde{K}_n \biggl( y_2,
\frac{y_2-y_1+1}{2};  x_2, \frac{x_2-x_1+1}{2} \biggr)
\nonumber\\[-8pt]\\[-8pt]
&&\qquad =f_1(x,y).\nonumber
\end{eqnarray}
Also, by~(\ref{dimerstoparticlesfour}) we have
%
\begin{eqnarray}
\qquad && \phi_{y_2,x_2} \biggl( \frac{y_2-y_1+1}{2},  \frac{x_2-x_1+1}{2} \biggr)\nonumber
\\
&&\qquad =
\frac{ \mathbb{I}_{y_2<x_2}}{2\pi i} \int_{\gamma_{r_1}} z^{((x_2-x_1)/2) -((y_2-y_1)/2)}
\frac{ (1-az)^{((y_2+1-x_2)/2)-1}}{(1+a/z)^{(y_2+1-x_2)/2}} \frac{dz}{z}\nonumber
\\
&&\qquad  =(-1)^{(y_2-x_2-1)/2} \frac{ \mathbb{I}_{y_2<x_2}}{2\pi i} \int
_{\gamma_{r_1}}z^{(y_1-x_1-1)/2} \frac{ (az-1)^{(y_2-x_2-1)/2}}{(z+a)^{(y_2-x_2+1)/2}} \,dz
\\
&&\qquad = (-1)^{(y_2-x_2-1)/2} \frac{ \mathbb{I}_{y_2<x_2}\mathbb
{I}_{y_1<x_1}}{2\pi i}\nonumber
\\
&&\qquad\quad{}\times  \int_{\gamma_{r_1}}
z^{(y_1-x_1-1)/2} \frac{ (az-1)^{(y_2-x_2-1)/2}}{(z+a)^{(y_2-x_2+1)/2}} \,dz\nonumber
\\
&&\qquad  = (-1)^{(y_2-x_2-1)/2} \frac{ \mathbb{I}_{y_2<x_2}\mathbb
{I}_{y_1<x_1}}{2\pi i}\nonumber
\\
&&\qquad\quad{}\times  \int_{\mathcal{E}_2}
z^{(y_1-x_1-1)/2} \frac{ (az-1)^{(y_2-x_2-1)/2}}{(z+a)^{(y_2-x_2+1)/2}} \,dz\nonumber
\\
&&\qquad  = (-1)^{(y_2-x_2-1)/2} \frac{\mathbb{I}_{y_1<x_1}}{2\pi i}
a^{(y_2-x_2-1)/2}\nonumber
\\
&&\qquad\quad{}\times \int _{\mathcal{E}_1} z^{(y_2-x_2-1)/2} \frac{ (1/a+z)^{(y_1-x_1-1)/2}}{(z+a+1/a)^{(y_2-x_2+1)/2}} \,dz. \nonumber
\end{eqnarray}
In the third equality, we use the fact that the integrand has no
singularity inside $\gamma_1$ if $y_1>x_1$ (and $y_1=x_1$ is not
possible). The fourth equality follows by deforming $\gamma_1$ to
$\mathcal{E}_2$ through infinity. The last equality follows since the
integrand has no singularity inside $\mathcal{E}_2$ if $y_2>x_2$ and
by making the shift $z \mapsto z+1/a$. We see that
%
\begin{eqnarray}
&& -(-1)^{(x_1-x_2+y_1-y_2+2)/4} \phi_{y_2,x_2} \biggl( \frac
{y_2-y_1+1}{2},
\frac{x_2-x_1+1}{2} \biggr)
\nonumber\\[-8pt]\\[-8pt]
&&\qquad  = \mathbb{I}_{x_1>y_1} f_2(x,y).\nonumber
\end{eqnarray}

\section{Asymptotics of dimers}\label{sec4}\label{secAsymptoticsofdimers}

In this section, we will give the proofs of the results on the local
asymptotics of the Aztec diamond. We start by proving Proposition \ref
{thinnedthickprop} about thinned and thickened determinantal point processes.

\begin{pf*}{Proof of Proposition \ref{thinnedthickprop}}
Let $\{y_j\}$ be the points of the determinantal point process with
kernel $K$, and let $\{n_j\}$ be independent Bernoulli random variables,
$\mathbb{P}[n_j=1]=\alpha$.
Let $\mathbb{E}_K$ denote the expectation for the determinantal point
process and $\mathbb{E}_n$ the expectation with respect to the
Bernoulli random
variables. Consider the thinned process. Then, by Fubini's theorem,
%
\begin{eqnarray}
\mathbb{E} \bigl[e^{-\sum_j\psi(x_j)} \bigr]& =&\mathbb{E}_n
\mathbb{E}_K \bigl[e^{-\sum_jn_j\psi(y_j)} \bigr]\nonumber
\\
&=&  \mathbb{E}_K
\mathbb{E}_n \biggl[\prod_{j}\bigl(1-
\bigl(1-e^{-n_j\psi
(y_j)}\bigr)\bigr) \biggr]\nonumber
\\
& =&\mathbb{E}_K \biggl[\prod_{j}
\bigl(1-\mathbb{E}_n\bigl[1-e^{-n_j\psi
(y_j)}\bigr]\bigr) \biggr]=
\mathbb{E}_K \biggl[\prod_{j}\bigl(1-
\alpha\phi(y_j)\bigr) \biggr]
\\
& =&\det(I-\phi\alpha K\mathbb{I}_A),
\end{eqnarray}
which proves (\ref{thinnedexpectation}). Next, let $\{m_j\}$ be
independent geometric random variables, $\mathbb{P}[m_j=k]=(1-\beta
)\beta^{k-1}$,
$k\ge1$, and let $\mathbb{E}_m$ denote the expectation with respect
to these random variables. Then
%
\begin{eqnarray}
\mathbb{E} \bigl[e^{-\sum_j\psi(x_j)} \bigr]& =&\mathbb{E}_m
\mathbb{E}_K \bigl[e^{-\sum_jm_j\psi(y_j)} \bigr]\nonumber
\\
&=& \mathbb{E}_K
\mathbb{E}_m \biggl[\prod_{j}\bigl(1-
\bigl(1-e^{-m_j\psi
(y_j)}\bigr)\bigr) \biggr]
\nonumber\\[-8pt]\\[-8pt]
& =&\mathbb{E}_K \biggl[\prod_{j}
\biggl(1-\frac{\phi(y_j)}{1-\beta
+\beta\phi(y_j)} \biggr) \biggr]\nonumber
\\
& =& \det\biggl(I-\frac{\phi}{1-\beta
+\beta\phi}K
\mathbb{I}_A \biggr),\nonumber
\end{eqnarray}
since
\begin{eqnarray}
\mathbb{E}_m \bigl[1-e^{-m_j\psi(y_j)} \bigr]& =&1-(1-\beta)
\sum_{k=1}^\infty\beta^{k-1}e^{-k\psi(y_j)}\nonumber
\\
& =&\frac{1-e^{-\psi(y_j)}}{1-\beta e^{-\psi(y_j)}}
\\
&=& \frac{\phi
(y_j)}{1-\beta+\beta\phi(y_j)}.\nonumber
\end{eqnarray}
This proves (\ref{thickenedexpectation}).\vadjust{\goodbreak}
\end{pf*}


\begin{pf*}{Proof of Theorem \ref{asymptoticsthmthinnedairy}}
By Lemma \ref{lemmakernelofsoutherndominoprocess} the south
domino process on the line $y=r$ is a determinantal point process with
kernel $L$ given by (\ref{asymptoticskernel}). Let us first consider
this process in a neighbourhood of the northern boundary, when
$r=[(1-k^2u(k))n]$, $k>0$. (Below we will often neglect the integer
part in this and in other expressions.
It is not difficult to see that this is unimportant.)
The kernel can be written
%
\begin{eqnarray}\label{L-kernel}
\qquad && L(x_1,x_2)
\nonumber\\[-8pt]\\[-8pt]
&&\qquad =-\frac{1}{(2\pi i)^2}\int
_{\Gamma_1}dz\int_{\Gamma
_2}dw\,\frac{w^{x_2-u(k)n}}{z^{x_1-u(k)n}}
\frac{1}{(a+w)(w-z)}e^{ng(z)-ng(w)},\nonumber
\end{eqnarray}
where
%
\begin{equation}
\label{gz} g(z)=\bigl(1-k^2u(k)\bigr)\log(a+z)+k^2u(k)
\log(az-1)-u(k)\log z.
\end{equation}
We have deformed the contours $\mathcal{E}_1$ and $\mathcal{E}_2$ to
new contours $\Gamma_1$ and $\Gamma_2$,
described below, which are good contours for the asymptotic analysis.
The argument in the logarithms is chosen in the interval $(0,2\pi)$.
We see that when $u(k)$ is given by~(\ref{asymptoticsuk}), $g(z)$
has a double zero at
%
\begin{equation}
\label{zc} z_c=\frac{1}{a+k\sqrt{1+a^2}}.
\end{equation}
We can now use a saddle-point argument to analyze the relevant
asymptotics of~(\ref{L-kernel}), and since this is a fairly standard
Airy kernel asymptotics
saddle point analysis, we will not go into all the details. For the
integration contours in (\ref{L-kernel}) we have chosen the steepest
descent contours given by the
level lines of the imaginary part of $g(z)$ starting
at $z_c$. It can be seen that we will have two ascending contours for
the real part of $g(z)$ which will leave in the directions $e^{\pm\pi
i/3}$ and
go to infinity. We can deform the contour $\mathcal{E}_2$ to a contour
$\Gamma_2$ consisting of these two pieces. We will have two descending
contours going from $z_c$ to $-a$ leaving in the directions
$e^{\pm2\pi i/3}$, and these can be combined into a contour $\Gamma
_1$; see Figure~\ref{figcontoursairy}.
If we have the scalings
%
\begin{eqnarray}
\label{x1x2} x_1&=&\bigl[u(k)n-\lambda n^{1/3}\xi\bigr],
\qquad
x_2 =\bigl[u(k)n-\lambda n^{1/3}\eta\bigr],
\end{eqnarray}
%
%
\begin{figure}[t]

\includegraphics{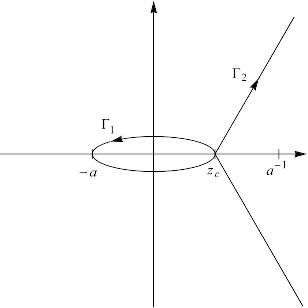}

\caption{A schematic diagram of the
contours of steepest ascent and descent for $g(z)$ for $z_c \in(0,1/a)$.}
\label{figcontoursairy}
\end{figure}%
then
%
\begin{eqnarray}
&& \lim_{n\to\infty} -\lambda n^{1/3}z_c^{x_1-x_2}
L(x_1,x_2)
\nonumber\\[-8pt] \label{scalinglimitL} \\[-8pt]
&&\qquad  = \alpha\frac{1}{(2\pi i)^2}\int
_{\Gamma} dz\int_{\Gamma} dw\,
\frac{1}{i(z+w)}e^{iz^3/3+i\xi z+iw^3/3+i\eta w}\nonumber
\\
&&\qquad  =\alpha K_{\Ai}(\xi,\eta)
\end{eqnarray}
uniformly for $\xi,\eta$ in a compact subset of $\mathbb{R}$, where
$\alpha=z_c/(z_c+a)$. Here $\Gamma$ is given by
$z(t)=-te^{(\pi-\theta)i}$, $t<0$ and $z(t)=te^{i\theta}$, $t\ge0$,
with a fixed $0<\theta<\pi/3$. Furthermore, $\lambda$ is given by
%
\begin{equation}
\label{lambda2} \lambda=z_c\bigl(-g^{(3)}(z_c)/2
\bigr)^{1/3}.
\end{equation}
A computation gives (\ref{lambda}).
To prove the result in Theorem \ref{asymptoticsthmthinnedairy} we
observe that, with $\phi=1-e^{-\psi}$ and $[n]=\{1,\ldots,n\}$,
%
\begin{eqnarray}
\qquad && \mathbb{E} \bigl[e^{-\sum_{j}\psi(\xi_j)} \bigr]
\nonumber\\[-8pt]\label{genfcn} \\[-8pt]
&&\qquad  = \mathbb
{E} \biggl[\prod
_j \biggl(1-\phi\biggl(\frac{nu(k)-x_j}{\lambda n^{1/3}} \biggr)
\biggr) \biggr]\nonumber
\\
&&\qquad  =\sum_{m=0}^n\frac{(-1)^m}{m!}\sum
_{x_1,\ldots,x_m\in[n]}\prod_{j=1}^m
\phi\biggl(\frac{nu(k)-x_j}{\lambda n^{1/3}} \biggr)\det\bigl(L(x_i,x_j)
\bigr)_{m\times m}.
\end{eqnarray}
Using the uniform convergence in (\ref{scalinglimitL}) and Hadamard's
inequality, we see that (\ref{genfcn}) converges to
%
\begin{eqnarray}
&& \sum_{m=0}^n\frac{(-1)^m}{m!}\int
_{\mathbb{R}^m}\prod_{j=1}^m
\phi(\xi_j)\det\bigl(\alpha K_{\Ai}(\xi_i,
\xi_j) \bigr)_{m\times m} d^m\xi
\nonumber\\[-8pt]\\[-8pt]
&&\qquad =\det(I-\phi\alpha
K_{\Ai}\mathbb{I}_A ),\nonumber
\end{eqnarray}
for every $\psi\in C_c^+(\mathbb{R})$, where $A=\operatorname
{supp}\phi=\operatorname{supp}\psi$; see, for example, \cite
{Joh05}. This proves the weak convergence claimed in the theorem; see,
for example,
\cite{DaleyVereJonesII}, page~138.

We turn now to the south domino process close to the southern boundary.
Take $r$ as before but with $k\in(-a^{-1}(1+a^2)^{1/2},-a(1+a^2)^{-1/2})$.
By Lemma \ref{lemmakernelofsoutherndominoprocess}
%
\begin{eqnarray}
\qquad\quad && L(x_1,x_2)
=-\frac{1}{(2\pi i)^2}\int_{\mathcal{E}_1}
dz\int_{\mathcal{E}_2} dw\,\frac{w^{x_2}}{z^{x_1}}\frac
{(a+z)^r(az-1)^{n-r}}{(a+w)^{r+1}(aw-1)^{n-r}}
\frac{1}{w-z}.
\end{eqnarray}
Deform $\mathcal{E}_2$ through infinity to a contour $\gamma_2$
containing $\mathcal{E}_1$ and $-a$ in its interior, but $1/a$
outside, to obtain
%
\begin{equation}
\qquad L(x_1,x_2)=\frac{1}{(2\pi i)^2}\int_{\mathcal{E}_1}
dz\int_{\gamma
_2} dw\,\frac{w^{x_2}}{z^{x_1}}\frac
{(a+z)^r(az-1)^{n-r}}{(a+w)^{r+1}(aw-1)^{n-r}}
\frac{1}{w-z}.
\end{equation}
%
\begin{figure}

\includegraphics{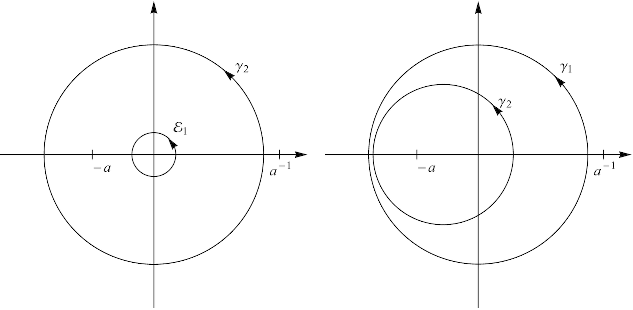}

\caption{The figure on the left represents deforming the contour
$\mathcal{E}_2$ to $\gamma_2$. The figure on the right represents
deforming the contour $\mathcal{E}_1$ through $\gamma_2$ to $\gamma
_1$. Note that this picks up a single integral term.}
\label{figcontourswitch}
\end{figure}%
Then move $\mathcal{E}_1$ to a contour $\gamma_1$ which surrounds
$\gamma_2$; see Figure~\ref{figcontourswitch}. This picks up a
contribution from the pole at $z=w$. Write
%
\begin{equation}
\label{Ltilde}
\qquad \widetilde{L}(x_1,x_2)=\frac{z_c^{x_1-x_2}}{(2\pi i)^2}\int
_{\gamma_1} dz\int_{\gamma_2} dw\,\frac{w^{x_2}}{z^{x_1}}
\frac
{(a+z)^r(az-1)^{n-r}}{(a+w)^{r+1}(aw-1)^{n-r}}\frac{1}{w-z}.
\end{equation}
Here $z_c$ is again given by (\ref{zc}). Note that now we have
$z_c\in(-\infty,-a)$. Thus we find
%
\begin{equation}
\label{Lrelation} z_c^{x_1-x_2}L(x_1,x_2)=
\frac{z_c^{x_1-x_2}}{2\pi i}\int_{\gamma
_1}\frac{z^{x_2-x_1}}{z+a} \,dz+
\widetilde{L}(x_1,x_2).
\end{equation}
The parameter $\beta$ of the thickened process is given by $\beta
=-a/z_c$, and we see that $0<\beta<1$. Set
%
\begin{equation}
M(x_1,x_2)=-\beta^{x_2-x_1}\mathbb{I}_{x_1<x_2}.
\end{equation}
Then (\ref{Lrelation}) gives
%
\begin{equation}
z_c^{x_1-x_2}L(x_1,x_2)=
\delta_{x_1,x_2}-M(x_1,x_2)+\widetilde{L}(x_1,x_2).
\end{equation}
A correlation kernel $L^\ast$ for the dual point process is given by
\cite{Bor02}, Proposition~4,
%
\begin{equation}
\label{Lstar} L^{\ast}(x_1,x_2)=\delta
_{x_1,x_2}-z_c^{x_1-x_2}L(x_1,x_2)=M(x_1,x_2)-
\widetilde{L}(x_1,x_2).
\end{equation}
If $x_1,x_2,\ldots$ are the points in the dual point process, we want
to look at
%
\begin{equation}
\mathbb{E} \biggl[\prod_j\exp\biggl(-\sum
_j\psi\biggl(\frac
{x_j-nu(k)}{\lambda n^{1/3}} \biggr) \biggr)
\biggr]=\det\bigl(I-fL^\ast g\bigr),
\end{equation}
where the Fredholm determinant is on $\ell^2(\{1,\ldots,n\})$ and so
is actually a determinant of a finite matrix.
Here we have introduced $f(x)=\phi((\lambda n^{1/3})^{-1}(x-u(k)n))$
and $g(x)=\mathbb{I}_A((\lambda n^{1/3})^{-1}(x-u(k)n))$, where $A$ is
the support of $\phi$.
Now, by~(\ref{Lstar}), we have
%
\begin{eqnarray}
\label{Fredholmdeterminantcomputation}
\det\bigl(I-fL^\ast g\bigr)& =&\det
(I-fMg+f\widetilde{L}g)
\\
& =&\det(I-fMg)\det\bigl(I+(I-fMg)^{-1}f\widetilde{L}\bigr)
\nonumber\\[-8pt]\\[-8pt]
&=& \det\Biggl(I+
\sum_{j=0}^n(fMg)^jf\widetilde{L}g
\Biggr)\nonumber
\\
& =&\det\Biggl(I+f\sum_{j=0}^n(Mf)^j
\widetilde{L}g \Biggr).
\end{eqnarray}
Here we have used that $gf=f$ and $fMg$ is nilpotent, $(fMg)^{n+1}=0$.
Hence, we also have $\det(I-fMg)=1$. Set
%
\begin{equation}
R=\sum_{j=1}^n(Mf)^j
\widetilde{L}.
\end{equation}
Then, by (\ref{Fredholmdeterminantcomputation}), we have
%
\begin{equation}
\label{Fredholmformula} \det\bigl(I-fL^\ast g\bigr)=\det\bigl
(I+f(R+\widetilde{L})g
\bigr).
\end{equation}
In order to prove the result in Theorem \ref
{asymptoticsthmthinnedairy} it suffices to show that, with the scaling
%
\begin{equation}
\label{Thickscaling} x_1=\bigl[u(k)n+\lambda n^{1/3}\xi\bigr],
\qquad x_2=\bigl[u(k)n+\lambda n^{1/3}\eta\bigr]
\end{equation}
we have that
%
\begin{equation}
\label{Rlimit} \lambda n^{1/3}(R+\widetilde{L}) (x_1,x_2)
\to-\frac{1}{1-\beta+\beta\phi
(\xi)}K_{\Ai}(\xi,\eta)
\end{equation}
uniformly for $\xi,\eta$ in a compact subset of $\mathbb{R}$ as
$n\to\infty$. We will prove this under the assumption that $\psi$ is also
continuously differentiable, which suffices to show the weak
convergence of the point process. This can be seen by an approximation argument.

We will make use of the following fact given below in~(\ref
{Kernelasymptotics}), which again is proved by a saddle point argument
very similar to the one discussed above. The only difference
is in the choice of contours.
As integration contours we will again choose level lines of the
imaginary part of $g(z)$ starting
at $z_c$. Recall that $z_c\in(-\infty, -a)$. It can be seen that we
will have two ascending contours for the real part of $g(z)$ which will
leave in the directions $e^{\pm\pi i/3}$ and
go to $0$. We combine these contours to a contour $\gamma_2'$ and use
it for our $w$-integration. We will have two descending contours going
from $z_c$ to $1/a$ leaving in the directions
$e^{\pm2\pi i/3}$. Combine them into a contour $\gamma_1'$, and use
it for the $z$-integration.

Let $x_1,x_2$ be as in (\ref{Thickscaling}), $r=[(1-k^2u(k))n]$ and
$k\in(-a(1+a^2)^{-1/2},\break -a^{-1}(1+a^2)^{1/2})$. Then, we have
%
\begin{eqnarray}\label{Kernelasymptotics}
\qquad && \frac{\lambda n^{1/3} z_c^{x_1-x_2}}{(2\pi
i)^2}\int_{\gamma_1'} dz\int
_{\gamma_2'} dw\, F(z) \frac{w^{x_2}}{z^{x_1}}\frac
{(a+z)^r(az-1)^{n-r}}{(a+w)^{r+1}(aw-1)^{n-r}}
\frac{1}{w-z}
\nonumber\\[-8pt]\\[-8pt]
&&\qquad \to-\frac
{D}{1-\beta}K_{\Ai}(\xi,\eta)\nonumber
\end{eqnarray}
uniformly for $\xi,\eta$ in compacts as $n\to\infty$, for $F(z)=1$
or $F(z)=\beta z_cf(x_1)/(z-\beta(1-f(x_1))z_c)$, where $D=F(z_c)$.
Note that with the scaling (\ref{Thickscaling}) we have $f(x_1)=\phi
(\xi)$ (ignoring integer parts).

By expanding $M$ and rearranging the sum, we have
\begin{eqnarray*}
&& R(x_1,x_2)
\\
&&\qquad  =\sum_{j=1}^n(Mf)^j
\widetilde{L}(x_1,x_2)
\\
&&\qquad  =\sum_{j=1}^n(-1)^j\sum
_{y_1,\ldots,y_j\in[n]}\beta^{y_j-x_1}\mathbb{I}_{(x_1<y_1)}
\cdots\mathbb{I}_{(y_{j-1}<y_j)}f(y_1)\cdots f(y_j)
\widetilde{L}(y_j,x_2)
\\
&&\qquad  =\sum_{j=1}^n(-1)^j\sum
_{t=j}^{n-x_1}\beta^tf(x_1+t)
\widetilde{L}(x_1+t,x_2)
\\
&&\quad\qquad{}\times \sum_{x_1<y_1<\cdots<y_{j-1},x_1+t}f(y_1)
\cdots f(y_{j-1}).
\end{eqnarray*}
Let $e_j(x_1,\ldots,x_{t-1})$ be the $j$th elementary symmetric
polynomial in $t-1$ variables, and write $f=(f(x_1+1),\ldots,f(x_1+t-1))$. By interchanging the sums, we obtain
%
\begin{eqnarray}
\label{Rsum} R(x_1,x_2)& =&-\sum
_{t=1}^{n-x_1}\beta^tf(x_1+t)
\widetilde{L}(x_1+t,x_2)\sum_{j=0}^{t-1}(-1)^{j-1}e_{j-1}(f)
\nonumber
\\[-8pt]
\\[-8pt]
& =&-\sum_{t=1}^{n-x_1}\beta^tf(x_1+t)
\widetilde{L}(x_1+t,x_2)\prod_{j=1}^{t-1}
\bigl(1-f(x_1+j)\bigr).
\nonumber
\end{eqnarray}
Set
%
\begin{equation}
\label{Tkernel} T(x_1,x_2)=-\sum
_{t=1}^{n-x_1}\beta^tf(x_1)
\bigl(1-f(x_1)\bigr)^{t-1}\widetilde{L}(x_1+t,x_2).
\end{equation}
We then have the following:

\begin{claim}\label{Rclaim}
For all $\xi,\eta$ in a compact subset of $\mathbb{R}$ there is a
constant $C$ such that if we have the scaling
(\ref{Thickscaling}), then
%
\begin{equation}
\label{RTestimate} \bigl|\lambda n^{1/3}(R-T) (x_1,x_2)\bigr|
\le\frac{C}{n^{1/3}}.
\end{equation}
\end{claim}

We proceed with the proof of the theorem and return to the proof of the
claim below. Using (\ref{Ltilde}) we see that
%
\begin{eqnarray}
T(x_1,x_2)& =&-\sum_{t=1}^{n-x_1}
\beta^tf(x_1) \bigl(1-f(x_1)
\bigr)^{t-1} \frac{z_c^{x_1-x_2}}{(2\pi i)^2}\nonumber
\\
&&\hspace*{6pt}{}\times \int_{\gamma_{1}}dz\int_{\gamma_{2}}dw \biggl(\frac{z_c}z \biggr)^t
\frac{w^{x_2}}{z^{x_1}}\frac
{(a+z)^r(az-1)^{n-r}}{(a+w)^{r+1}(aw-1)^{n-r}}\frac{1}{w-z}
\\
& =&-\frac{f(x_1)}{1-f(x_1)}\frac{z_c^{x_1-x_2}}{(2\pi i)^2}\nonumber
\\
&&\hspace*{6pt}{}\times \int_{\gamma
_1}dz\int _{\gamma_2}dw\,\frac{\beta(1-f(x_1))z_c/z-(\beta
(1-f(x_1))z_c/z)^{n-x_1+1}}{
1-\beta(1-f(x_1))z_c/z}\nonumber
\\
&&\hspace*{76pt}{} \times\frac{w^{x_2}}{z^{x_1}}\frac
{(a+z)^r(az-1)^{n-r}}{(a+w)^{r+1}(aw-1)^{n-r}}\frac{1}{w-z}\nonumber
\\
& =&-f(x_1)\frac{z_c^{x_1-x_2}}{(2\pi i)^2}\int_{\gamma_1}dz\int
_{\gamma_2}dw\,\frac{\beta z_c}{z-\beta(1-f(x_1))z_c}\frac{w^{x_2}}{z^{x_1}} \nonumber
\\
&&\hspace*{126pt}{}\times
\frac{(a+z)^r(az-1)^{n-r}}{(a+w)^{r+1}(aw-1)^{n-r}}\frac{1}{w-z}.\nonumber
\end{eqnarray}
The term involving $(\beta(1-f(x_1))z_c/z)^{n-x_1+1}$ does not
contribute since the \mbox{$z$-}integral integrates to $0$ by Cauchy's theorem
because the integrand is analytic outside $\gamma_1$ including at
$\infty$.

It now follows from (\ref{Kernelasymptotics}) that
%
\begin{equation}
\lambda n^{1/3}T(x_1,x_2)\to\frac{\beta\phi(\xi)}{1-\beta(1-\phi
(\xi))}
\frac{1}{1-\beta}K_{\Ai}(\xi,\eta)
\end{equation}
uniformly as $n\to\infty$. If we combine this with (\ref
{RTestimate}) and the fact, again by (\ref{Kernelasymptotics}) with
$F(z)=1$, that
%
\begin{equation}
\lambda n^{1/3}\widetilde{L}(x_1,x_2)\to-
\frac{1}{1-\beta}K_{\Ai}(\xi,\eta)
\end{equation}
we obtain (\ref{Rlimit}), which is what we wanted to prove.

It remains to prove the claim. We will use the following two facts. If
$\xi,\eta$ belongs to a compact subset,
there is a constant $B>0$ and a constant $C$ such that if we
have the scaling (\ref{x1x2}),
then: (i) $f(x_1+t)=0$ if $t\notin[-Bn^{1/3},Bn^{1/3}]$ and (ii)
$|\lambda n^{1/3}\widetilde{L}(x_1+t,x_2)|\le C$ for all $t\in
[-Bn^{1/3},Bn^{1/3}]$.
That (i) holds follows immediately from the fact that $\phi$ has
compact support, and (ii) follows from (\ref{Kernelasymptotics}) in
the case
$F(z)=1$ since we have uniform convergence. We will also use the inequality
\[
\Biggl\llvert\prod_{j=1}^n
a_j-\prod_{j=1}^nb_j
\Biggr\rrvert\le\sum_{j=1}^n|a_j-b_j|,
\]
provided $|a_j|,|b_j|\le1$ for all $j$. This is easy to prove by induction.\

Now, since $\phi$ is continuously differentiable and has compact support,
%
\begin{eqnarray}
&& \Biggl\llvert f(x_1+t)\prod_{j=1}^{t-1}
\bigl(1-f(x_1+j)\bigr)-f(x_1)\prod
_{j=1}^{t-1}\bigl(1-f(x_1)\bigr)\Biggr
\rrvert\nonumber
\\
&&\qquad \le\sum_{j=1}^{t}\bigl|f(x_1+j)-f(x_1)\bigr|\nonumber
\\
&&\qquad  =\sum_{j=1}^{t}\biggl\llvert\phi
\biggl(\frac{[u(k)n+\lambda n^{1/3}\xi
]+j-u(k)n}{\lambda n^{1/3}} \biggr)
\\
&&\hspace*{57pt}{} - \phi\biggl(\frac{[u(k)n+\lambda
n^{1/3}\xi]-u(k)n}{\lambda
n^{1/3}} \biggr)\biggr
\rrvert\nonumber
\\
&&\qquad \le\sum_{j=1}^{t} \frac{Cj}{n^{1/3}}\le
\frac{Ct^2}{n^{1/3}},\nonumber
\end{eqnarray}
for some constant $C$. Thus
%
\begin{eqnarray}
&& \bigl|\lambda n^{1/3}(R-T) (x_1,x_2)\bigr| \nonumber
\\
&&\qquad \le\sum_{t=1}^{n-x_1}\beta^t\bigl|f(x_1+t)\bigr|\bigl|
\lambda n^{1/3}\widetilde{L}(x_1+t,x_2)\bigr|
\nonumber\\[-8pt]\\[-8pt]
&&\hspace*{50pt}{}\times  \Biggl
\llvert f(x_1+t)\prod_{j=1}^{t-1}
\bigl(1-f(x_1+t)\bigr)-f(x_1)\prod
_{j=1}^{t-1}\bigl(1-f(x_1)\bigr)\Biggr
\rrvert\nonumber
\\
&&\qquad \le\frac{C}{n^{1/3}}\sum_{t=1}^{n-x_1}
\beta^tt^2\le\frac
{C}{n^{1/3}},
\nonumber
\end{eqnarray}
since $0<\beta<1$. In the next to last inequality we used (i) and (ii)
above. The estimate (ii) works since we can use (i) to restrict the
range of $t$-values.
\end{pf*}

Above we have been concerned with the behavior of the south dominoes at
the boundary of the frozen region.
We can also consider the behavior of the south dominoes as we enter the
bulk but still stay close to the boundary at a macroscopic scale.
We then get as a scaling limit the thinned sine kernel point process
with the same parameter $\alpha$. We will not go into the details.

Next we will give the proof of Theorem \ref{airytopoisson} which is
concerned with the case when $a$ grows with $n$ but not too fast.
In the case when $a$ instead goes to zero with $n$ but not to fast we
should have convergence to the Airy kernel point process at the
northern boundary for the south
domino process. Note that the thinning parameter $\alpha\to1$ as
$a\to0$. If $a=\gamma/n$ for some fixed $\gamma>0$, then we expect instead
the discrete Bessel kernel in the limit as $n\to\infty$. We will not
go into the details on how to prove these assertions.

\begin{pf*}{Proof of Theorem \ref{airytopoisson}}
Set
%
\begin{equation}
\label{mathcalL} \mathcal{L}(x_1,x_2)=z_c^{x_1-x_2}L(x_1,x_2),
\end{equation}
with $L$ as in (\ref{L-kernel}) and $z_c$ given by (\ref{zc}). Take
$\phi\in C_c(\mathbb{R})$, $0\le\phi\le1$, and let
%
\begin{equation}
\label{fj} f(j)=\phi\biggl(\frac{nu(k)+c(a)n^{1/3}-j}{c(a)n^{1/3}}
\biggr)
\end{equation}
for $j\in\mathbb{Z}$. Furthermore, define
%
\begin{eqnarray}
\label{Mn}
M_n(\xi,\eta)
&=&-c(a)n^{1/3}\mathcal{L}\bigl(
\bigl[u(k)n-c(a)n^{1/3}(\xi-1)\bigr],
\nonumber\\[-8pt]\\[-8pt]
&&\hspace*{62pt} \bigl[u(k)n-c(a)n^{1/3}(
\eta-1)\bigr]\bigr).\nonumber
\end{eqnarray}
Set
%
\begin{equation}
\label{Ij} I_j= \biggl(\frac{u(k)n+c(a)n^{1/3}-(j+1)}{c(a)n^{1/3}},\frac
{u(k)n+c(a)n^{1/3}-j}{c(a)n^{1/3}} \biggr],
\end{equation}
for $j\in\mathbb{Z}$, and for $\xi\in I_j$, define
$\tilde{\phi}_n(\xi)=f(j)$.

Let $A$ be a compact subset of the real line such that $\operatorname
{supp}\tilde{\phi}_n\subseteq A$ for all $n$. Then
%
\begin{eqnarray}\label{Poissonexpectation}
&& \mathbb{E} \biggl[\prod_j\bigl(1-
\phi(\xi_j)\bigr) \biggr]
\nonumber\\[-8pt]\\[-8pt]
&&\qquad  =\sum_{m=0}^n
\frac{(-1)^m}{m!}\sum_{x_1,\ldots,x_m\in[n]}\prod
_{j=1}^m f(x_j)\det\bigl(
\mathcal{L}(x_i,x_j)\bigr)_{m\times m}\nonumber
\\
&&\qquad  =\sum_{m=0}^n\frac{(-1)^m}{m!}\int
_{\mathbb{R}^m}\prod_{j=1}^m
\tilde{\phi}_n(\xi_j)\det\bigl(M_n(
\xi_i,\xi_j)\bigr)_{m\times m} d^m\xi
\\
&&\qquad  =\det(I-\tilde{\phi}_n M_n\mathbb{I}_A)_{L^2(\mathbb{R})},
\end{eqnarray}
where $[n]=\{1,\ldots,n\}$. The second equality follows from (\ref
{mathcalL}) to (\ref{Ij}) and the fact that the integrand is constant
on the intervals $I_j$
in each variable.
Assume that we can show that
%
\begin{equation}
\label{HilbertSchmidt} \|\tilde{\phi}_n M_n\mathbb{I}_A
\|_2\to0
\end{equation}
as $n\to\infty$, where $\|\cdot\|_2$ is the Hilbert--Schmidt norm, and
%
\begin{equation}
\label{Mntrace} M_n(\xi,\xi)\to\sqrt{(1-\xi)_+}
\end{equation}
uniformly for $\xi\in A$ as $n\to\infty$. The result then follows
from (\ref{Poissonexpectation}) in the following way. If the operators
$B_n$ on $L^2(\mathbb{R})$ are
trace class then the determinant $\det_2(I+B_n)$ and the Fredholm
determinant $\det(I+B_n)$ are related by
%
\begin{equation}
\det(I+B_n)=e^{\operatorname{tr}B_n}{\det}_{2}(I+B_n).
\end{equation}
Also, if we have $\|B_n\|_2\to0$ as $n\to\infty$, then we have $\det
_2(I+B_n)\to1$ as \mbox{$n\to\infty$;} see \cite{GGK00}. Hence
%
\begin{equation}
\label{Fredholdeterminantidentity} \det(I-\tilde{\phi}_n M_n
\mathbb{I}_A)=e^{\operatorname{tr}\tilde{\phi}_n M_n\mathbb{I}_A}{\det}_{2}(I-\tilde{\phi}_n M_n\mathbb{I}_A).
\end{equation}
It follows from (\ref{HilbertSchmidt}) that $\det_2(I-\tilde{\phi}_n M_n1_A)\to0$ as $n\to\infty$, and from (\ref{Mntrace}) that
%
\begin{equation}
\operatorname{tr}\tilde{\phi}_n M_n1_A=
\int_{\mathbb{R}}\tilde{\phi}_n(\xi)M_n(\xi,
\xi) \,d\xi\to\int_{\mathbb{R}}\phi(\xi)\sqrt{(1-\xi)_+} \,d\xi,
\end{equation}
as $n\to\infty$. Thus, by (\ref{Poissonexpectation}) and (\ref
{Fredholdeterminantidentity}),
%
\begin{equation}
\lim_{n\to\infty}\mathbb{E} \biggl[\prod
_j\bigl(1-\phi(\xi_j)\bigr)
\biggr]=e^{-\int_{\mathbb{R}}\phi(\xi)\sqrt{(1-\xi)_+} \,d\xi},
\end{equation}
which is what we wanted to prove.

We turn now to the asymptotic analysis of $M_n$ and the proof of (\ref
{HilbertSchmidt}) and (\ref{Mntrace}).
We will denote by $C$ a generic constant that can depend on $k$ and $d$
in Claim \ref{claimAiryapproximation} but not on $n$ or $a$.
Let $\lambda$ be given by (\ref{lambda}) and
define
%
\begin{equation}
M_n^{(1)}(\xi,\eta)=-\lambda n^{1/3}\mathcal{L}
\bigl(\bigl[nu(k)-\lambda n^{1/3}\xi\bigr],\bigl[nu(k)-\lambda
n^{1/3}\eta\bigr]\bigr).
\end{equation}

%
\begin{claim}\label{claimAiryapproximation}
Let $\alpha$ be given by (\ref{alpha}) and fix $d>0$. Then
%
\begin{equation}
\label{Airyapproximation} \bigl|M_n^{(1)}(\xi,\eta)-\alpha
K_{\Ai}(\xi,\eta)\bigr|\le\frac{C}{a^2}
\end{equation}
for all $\xi,\eta\in[-da^{4/3},da^{4/3}]$.
\end{claim}

Before proving the claim we finish the proof of the theorem. Set
$\tilde{c}(a)=c(a)/\lambda$, and note that $\tilde{c}(a)\sim\pi
^{2/3}a^{4/3}(1+k)^{2/3}$ as $a\to\infty$.
Then we find that
%
\begin{equation}
\label{Mnrelation} M_n(\xi+1,\eta+1)=\tilde{c}(a)M_n^{(1)}
\bigl(\tilde{c}(a)\xi,\tilde{c}(a)\eta\bigr).
\end{equation}
Thus, with an appropriate fixed $d_1>0$, we have
%
\begin{eqnarray}
\label{HSestimate} \|\tilde{\phi}_n M_n\mathbb{I}_A
\|_2& =&\int_{\mathbb{R}^2}\tilde{\phi}_n(x)^2M_n(x,y)^2
\mathbb{I}_A(y) \,dx\,dy\nonumber
\\
&\le &\int_{[-d_1,d_1]^2}M_n(
\xi+1,\eta+1)^2 \,d\xi \,d\eta
\nonumber
\\[-8pt]
\\[-8pt]
& =&\tilde{c}(a)^2\int_{[-d_1,d_1]^2}M_n^{(1)}
\bigl(\tilde{c}(a)\xi,\tilde{c}(a)\eta\bigr)^2 \,d\xi \,d\eta\nonumber
\\
&=& \int
_{[-da^{4/3},da^{4/3}]^2}M_n^{(1)}(\xi,\eta)^2
\,d\xi \,d\eta,
\nonumber
\end{eqnarray}
where $d=d_1\pi^{2/3}(1+k)^{2/3}$. Using (\ref{Airyapproximation}) we
see that
%
\begin{equation}
\label{HSestimate2} \qquad\quad \|\tilde{\phi}_n M_n\mathbb{I}_A
\|_2\le2\alpha^2\int_{[-da^{4/3},da^{4/3}]^2}K_{\Ai}(
\xi,\eta)^2 \,d\xi \,d\eta+C d^2\bigl(a^{4/3}a^{-2}
\bigr)^2.
\end{equation}
Now $a^{-4/3}\to0$, since $a=a(n)\to\infty$ as $n\to\infty$, and
we see that in order to prove (\ref{HilbertSchmidt}) it remains to
control the integral
in (\ref{HSestimate2}). If we use the identities
%
\begin{equation}
\label{Airykernelidentity} \int_{-\infty}^\infty
K_{\Ai}(x,y)^2 \,dy=K_{\Ai}(x,x),
\end{equation}
and
%
\begin{equation}
\label{Airykernelidentity2} K_{\Ai}(x,x)=\oAi'(x)^2-x
\oAi(x)^2
\end{equation}
we see that
%
\begin{eqnarray}
&& 2\alpha^2\int_{[-da^{4/3},da^{4/3}]^2}K_{\Ai}(\xi,
\eta)^2 \,d\xi \,d\eta\nonumber
\\
&&\qquad \le2\alpha^2\int_{-da^{4/3}}^\infty
\biggl(\int_{-\infty}^\infty K_{\Ai}(\xi,
\eta)^2 \,d\eta\biggr) \,d\xi
\nonumber
\\
&&\qquad  =2\alpha^2\int_{-da^{4/3}}^{\infty}K_{\Ai}(
\xi,\xi) \,d\xi=2\alpha^2\int_{-da^{4/3}}^\infty
\oAi'(\xi)^2-\xi\oAi(\xi)^2 \,d\xi
\\
&&\qquad  =\frac{2\alpha^2}3 \bigl[2\bigl(da^{4/3}\bigr)^2
\oAi^2\bigl(-da^{4/3}\bigr)\nonumber
\\
&&\hspace*{53pt}{} +2da^{4/3}\oAi
'\bigl(-da^{4/3}\bigr)^2-\oAi
\bigl(-da^{4/3}\bigr)\oAi'\bigl(-da^{4/3}\bigr)
\bigr].
\nonumber
\end{eqnarray}
Now, as $r\to\infty$ we have
%
\begin{eqnarray}
\oAi(-r)&=&\frac{1}{\sqrt{\pi}}r^{-1/4}\sin\biggl(\frac{2}3
r^{3/2}+\frac{\pi}4 \biggr)+\cdots,
\nonumber\\[-8pt]\\[-8pt]
\oAi'(-r)&=&
\frac{1}{\sqrt{\pi}}r^{1/4}\cos\biggl(\frac{2}3
r^{3/2}+\frac{\pi}4 \biggr)+\cdots\nonumber
\end{eqnarray}
and hence, since $\alpha\sim(1+k)^{-1}a^{-2}$ for large $a$, we
obtain the bound
%
\begin{equation}
2\alpha^2\int_{[-da^{4/3},da^{4/3}]^2}K_{\Ai}(\xi,
\eta)^2 \,d\xi \,d\eta\le Ca^{-4}\bigl(da^{4/3}
\bigr)^{3/2}\le Ca^{-2},
\end{equation}
and since $a=a(n)\to\infty$ we have proved (\ref{HilbertSchmidt}).

We now turn to the proof of (\ref{Mntrace}).
It follows from (\ref{Airyapproximation}) and (\ref{Mnrelation}) that
%
\begin{eqnarray}
\bigl|M_n(\xi,\xi)-\alpha\tilde{c}(a)K_{\Ai}\bigl(\tilde{c}(a)
(\xi-1),\tilde{c}(a) (\xi-1)\bigr)\bigr|&\le& Ca^{4/3}a^{-2}
\nonumber\\[-8pt]\\[-8pt]
&=&Ca^{-2/3}\nonumber
\end{eqnarray}
for all $\xi$ in a compact interval, and since $a\to\infty$ as $n\to
\infty$, (\ref{Mntrace}) follows by using (\ref{Airykernelidentity2})
and standard asymptotic formulas for the Airy function and its derivative.

It remains to prove Claim \ref{claimAiryapproximation}. Let $\gamma
>0$ be given by $\gamma^3g^{(3)}(z_c)=-2$.
From~(\ref{lambda2}) we see that that $\lambda=z_c/\gamma$. Let $g(z)$ be defined by (\ref{gz}) and write
%
\begin{equation}
f_{\xi}(\zeta)=\lambda n^{1/3}\xi\log\biggl(1+
\frac{\gamma\zeta
}{z_c n^{1/3}} \biggr)
\end{equation}
and
%
\begin{equation}
F_{\xi}(\zeta)=f_{\xi}(\zeta)+n\bigl(g\bigl(z_c+
\gamma\zeta n^{-1/3}\bigr)-g(z_c)\bigr).
\end{equation}
If we use (\ref{L-kernel}) and make the change of variables
$z=z_c+\gamma\zeta n^{-1/3}$, $w=z_c+\gamma\omega n^{-1/3}$
we obtain
%
\begin{eqnarray}\label{MN1formula}
M_n^{(1)}(\xi,\eta)
&=& \frac{\alpha}{(2\pi i)^2}\int
_{\mathcal{C}} d\zeta\int_{\mathcal{D}} d\omega\,
\frac{a+z_c}{(a+z_c+\gamma\omega
n^{-1/3})(\omega-\zeta)}
\nonumber\\[-8pt]\\[-8pt]
&&\hspace*{87pt}{}\times  e^{F_\xi(\zeta)-F_\eta(\omega)},\nonumber
\end{eqnarray}
where $\mathcal{C}$ and $\mathcal{D}$ are the images of the steepest
descent contours in Theorem \ref{asymptoticsthmthinnedairy}.
If $|\omega|\ge a^2n^{1/3}$ we have the estimate
%
\begin{equation}
\label{infinityestimate} \bigl\llvert e^{-F_\eta(\omega)}\bigr\rrvert
\le
C^na^{2n}\biggl\llvert\frac
{n^{1/3}}{a^{2/3}\omega}\biggr\rrvert
^{n/2}
\end{equation}
for all sufficiently large $n$. $\mathcal{C}$ will lie inside $|\zeta
|\le a^2n^{1/3}$, and using
(\ref{infinityestimate}) and the estimates we describe below for the
$\zeta$-integration,\vspace*{1pt} we see that we can replace $\mathcal{D}$
by the part of $\mathcal{D}$ that lies inside $|\omega|\le
a^2n^{1/3}$. We denote this part by $\mathcal{D}$ also for simplicity.

Let $\mathcal{C}_1^\ast$ be the part of $\mathcal{C}$ in the disk
$|\zeta|\le n^{1/15}$, $\mathcal{C}_2^\ast$ the part in the annulus
$n^{1/15}\le|\zeta|\le n^{7/45}$ and $\mathcal{C}_3^\ast$ the part
in $|\zeta|\ge n^{7/45}$. Let $\mathcal{C}_i$ and $\widebar{\mathcal{C}}_i$ be the parts of
$\mathcal{C}_i^\ast$ that lie in the upper and lower half plane,
respectively. We make the analogous definitions for $\mathcal{D}$. We
will consider
estimates of $F_\xi(\zeta)$ on $\mathcal{C}_i$, $i=1,2,3$. The
estimates on $\widebar{\mathcal{C}}_i$ are the same by symmetry, and the
estimates on $\mathcal{D}$
are analogous. The estimates that we need are
%
\begin{equation}
\label{gestimate1} \operatorname{Re} n\bigl(g\bigl(z_c+\gamma\zeta
n^{-1/3}\bigr)-g(z_c)\bigr)\le-\tfrac{1}6|
\zeta|^3
\end{equation}
for all $\zeta\in\mathcal{C}_1+\mathcal{C}_2$ and
%
\begin{equation}
\label{gestimate2} \operatorname{Re} n\bigl(g\bigl(z_c+\gamma\zeta
n^{-1/3}\bigr)-g(z_c)\bigr)\le-\tfrac{1}6n^{7/15}
\end{equation}
for all $\zeta\in\mathcal{C}_3$ and $n$ sufficiently large. To see
this note that we can write
%
\begin{equation}
n\bigl(g\bigl(z_c+\gamma\zeta n^{-1/3}
\bigr)-g(z_c)\bigr)=-\tfrac{1}3\zeta^3+h_1(
\zeta),
\end{equation}
where
%
\begin{equation}
h_1(\zeta)=\frac{\gamma^4}{6n^{1/3}}\int_0^\zeta
g^{(4)}\biggl(z_c+\frac
{\gamma s}{n^{1/3}}\biggr) (
\zeta-s)^3 \,ds.
\end{equation}
Now, if we have $|\zeta|\le n^{7/45}$, then
%
\begin{equation}
\label{h1estimate} \bigl|h_1(\zeta)\bigr|\le Ca^{2/3}n^{-1/3}|
\zeta|^4.
\end{equation}
The estimate (\ref{h1estimate}) follows from the fact that
%
\begin{equation}
\biggl| g^{(4)}\biggl(z_c+\frac{\gamma s}{n^{1/3}}\biggr)\biggr|\le
Ca^2
\end{equation}
for $|s|\le n^{7/45}$ since $\gamma\le Ca^{-1/3}$. If we have $\zeta
\in\mathcal{C}_1+\mathcal{C}_2$, then
$0=-\frac{1}3\operatorname{Im}\zeta^3+\operatorname{Im}h_1(\zeta)$,
and if we write $\zeta=r e^{i\theta}$, this gives
$r^3\sin3\theta=3\operatorname{Im}h_1(re^{i\theta})$ and hence, by
(\ref{h1estimate}),
%
\begin{equation}
|\sin3\theta|\le Ca^{2/3}n^{-1/3} r\le Cn^{-1/9}.
\end{equation}
Since $\mathcal{C}_1$ leaves $z_c$ in the direction $e^{2\pi i/3}$, we
must have $\cos3\theta\ge2/3$ for all large~$n$. Thus
%
\begin{eqnarray}
&& \operatorname{Re} n\bigl(g\bigl(z_c+\gamma\zeta n^{-1/3}
\bigr)-g(z_c)\bigr)
\nonumber\\[-8pt]\\[-8pt]
&&\qquad  =-\tfrac{1}3|\zeta|^3\cos3
\theta+\operatorname{Re} h_1(\zeta)\nonumber
\\
&&\qquad \le -\tfrac{2}9|\zeta|^3\bigl(1-Ca^{2/3}n^{-1/3}|
\zeta|\bigr)\le-\tfrac{1}6|\zeta|^3
\end{eqnarray}
if $\zeta\in\mathcal{C}_1+\mathcal{C}_2$ and $n$ is large. This
proves (\ref{gestimate1}). Since $\operatorname{Re} n(g(z_c+\gamma
\zeta n^{-1/3})-g(z_c))$
is decreasing as we move along $\mathcal{C}$ in the upper half plane
starting at $z_c$, we see that the value on $\mathcal{C}_3$ must be
$\le-n^{7/15}/6$
by using the estimate (\ref{gestimate1}) at the point where $\mathcal
{C}$ meets $|\zeta|=n^{7/15}$ (the endpoint of $\mathcal{C}_2$). This
proves~(\ref{gestimate2}).

We can write
%
\begin{equation}
f_\xi(\zeta)=\xi\zeta+\xi h_2(\zeta),
\end{equation}
where
%
\begin{equation}
h_2(\zeta)=-\frac{1}{\lambda n^{1/3}}\int_0^\zeta
\frac{\zeta
-s}{(1+s/\lambda n^{1/3})^2} \,ds,
\end{equation}
and we see that if $|\zeta|\le n^{7/45}$, then
%
\begin{equation}
\label{h2estimate} \bigl|h_2(\zeta)\bigr|\le Ca^{2/3}n^{-1/3}|
\zeta|^2.
\end{equation}
We start by estimating $F_\xi(\zeta)$ on $\mathcal{C}_3$. If $\zeta
\in\mathcal{C}_3$, then $|\zeta|\le a^2n^{1/3}$ and hence
%
\begin{equation}
\operatorname{Re} f_\xi(\zeta)=\lambda\xi n^{1/3}\log\biggl
\llvert1+\frac{\zeta}{\lambda n^{1/3}}\biggr\rrvert\le Cn^{6/15}\log n,
\end{equation}
and combining this with (\ref{gestimate2}), we see that $\operatorname
{Re} F_\xi(\zeta)\le-\frac{1}{12}n^{7/15}$ for large $n$,
and hence the contribution from $\mathcal{C}_3$ is negligible.\vspace*{1pt}

Since $|\xi|\le Ca^{4/3}$ and $a\le Cn^{1/10}$ we see that if $\zeta
\in\mathcal{C}_2$, then $\operatorname{Re} f_\xi(\zeta)\le
Cn^{1/10}|\zeta|$. Combining
this with (\ref{gestimate1}) we see that $\operatorname{Re} F_\xi
(\zeta)\le-|\zeta|^3/12$ for $n$ large, and hence
the contribution from $\mathcal{C}_2$ is negligible.

Thus, with an error that is much smaller than $Ca^{-2}$, we can replace
$M_n^{(1)}(\xi,\eta)$ by
%
\begin{eqnarray}\label{Mn1tildeestimate}
\qquad \widetilde{M}_n^{(1)}(\xi,\eta)\nonumber
&=& \frac{\alpha}{(2\pi i)^2}
\int_{\mathcal{C}_1+\widebar{\mathcal{C}}_1} d\zeta
\int_{\mathcal
{D}_1+\widebar{\mathcal{D}}_1}
d\omega\,\frac{a+z_c}{(a+z_c+\gamma\omega n^{-1/3})(\omega-\zeta)}
\nonumber\\[-8pt]\\[-8pt]
&&\hspace*{126pt}{}\times
e^{F_\xi
(\zeta)-F_\eta(\omega)}.\nonumber
\end{eqnarray}
Set $\delta_n=a(n)/n^{1/10}$. By assumption $\delta_n\to0$ as $n\to
\infty$.
If $|\zeta|\le n^{1/15}$, it follows from (\ref{h1estimate}) and
$(\ref{h2estimate})$ that
$F_\xi(\zeta)=-\zeta^3/3+\xi\zeta+r_n(\zeta)$, where $|r_n(\zeta
)|\le C\delta_n^{2/3}$. Also
%
\begin{equation}
\alpha\biggl\llvert\frac{a+z_c}{(a+z_c+\gamma\omega n^{-1/3})}-1\biggr
\rrvert\le C/a^2,
\end{equation}
if $\omega\in\mathcal{D}_1+\widebar{\mathcal{D}}_1$, and thus we can
approximate $\widetilde{M}_n^{(1)}(\xi,\eta)$ with
%
\begin{equation}
\label{approximateAirykernel} \frac{\alpha}{(2\pi i)^2}\int_{\mathcal
{C}_1+\widebar{\mathcal{C}}_1} d\zeta\int
_{\mathcal{D}_1+\widebar{\mathcal{D}}_1} d\omega\, e^{-\zeta^3/3+\xi\zeta
+\omega/3-\eta\omega}\frac{1}{\omega-\zeta}.
\end{equation}
Note that, if $|\zeta|\ge n^{1/15}$, then $\xi|/|\zeta|^2\le
Ca^{4/3}/n^{2/15}\le C\delta_n^{4/3}$ and thus,
with a negligible error, we can replace the expression in (\ref
{approximateAirykernel}) by $\alpha K_{\Ai}(\xi,\eta)$.
This completes the proof of the claim and the theorem.
\end{pf*}

\section{Gibbs measure}\label{sec5}\label{secBulk}

In this section, we continue our study of the asymptotics of domino
tilings with the study of the local Gibbs measure. We denote the
asymptotic coordinates by $\xi=(\xi_1,\xi_2)$. That is, for a
vertex inside the Aztec diamond denoted by $x=(x_1,x_2)$, we have
$x/(2n)\to\xi$.
For the remaining calculations of this paper, we use the same saddle
point function. This saddle point function is an extension of~(\ref
{gz}) because we now keep track of the asymptotic coordinates. Define
%
\begin{eqnarray}\label{bulkdefg}
g(z;\xi) &:=& g(z;\xi_1,\xi_2)
:=
\xi_2 \log(a+z) +(1-\xi_2) \log(az -1) - \xi_1
\log z.\hspace*{-25pt}
\end{eqnarray}
Recall that $\mathcal{D}$ denotes the unfrozen region and is given by
the area bounded by the ellipse
%
\begin{equation}
\frac{(v-u)^2}{1-p} +\frac{(u+v-1)^2}{p}=1,
\end{equation}
where $p=1/(1+a^2)$.

%
\begin{lemma} \label{bulklemmap}
The equation $g'(z;\xi_1,\xi_2)=0$ has a unique solution $z=z_\xi$
in $\mathbb{H}$ if and only if $\xi\in\mathcal{D}$. 
\end{lemma}

\begin{pf}
We expand out the equation $g'(z;\xi_1,\xi_2)=0$. We find that this
is equal to
%
\begin{equation}
-\frac{\xi_1}{z}+\frac{a (1-\xi_2 )}{a z-1}+\frac{\xi
_2}{a+z} =0.
\end{equation}
We solve the above equation with respect to $z$ and the solutions are
given by
%
\begin{eqnarray}
\quad&& \bigl( a^2 \xi_1-\xi_1 + \xi_2 + a^2 \xi_2-a^2
\nonumber\\[-8pt]\\[-8pt]
&&\qquad{}   \pm\sqrt{-4a^2(1-\xi_1) \xi_1+ \bigl(\xi_2-\xi_1 +a^2 (\xi_2+\xi_1-1)\bigr)^2}\bigr)
/ \bigl(2 a(1-\xi_1)\bigr).\nonumber
\end{eqnarray}
The expression under the square root term in the above equation is
given by
%
\begin{eqnarray}
&&-4a^2(1-\xi_1) \xi_1+ \bigl(
\xi_2-\xi_1 +a^2 (\xi_2+
\xi_1-1)\bigr)^2
\nonumber
\\[-8pt]
\\[-8pt]
&&\qquad  = \biggl( \frac{(\xi_1-\xi_2)^2}{a^2} + (\xi_2+\xi_1-1
)^2 \biggr)a^2\bigl(1+a^2\bigr)
-a^2,
\nonumber
\end{eqnarray}
which is less than zero if and only if $\xi_1, \xi_2 \in\mathcal
{D}$. Therefore, we set
%
\begin{eqnarray}\label{bulkdefzxi}
\qquad z_{\xi} &=&
\bigl( a^2 \xi_1-\xi_1 + \xi_2 + a^2 \xi
_2-a^2\nonumber
\\
&&\hspace*{3pt}{} +i \sqrt
{4a^2(1-\xi_1) \xi_1- \bigl(\xi_2-\xi_1 +a^2 (\xi_2+\xi_1-1)\bigr)^2}\bigr)
\\
&&{} / \bigl(2 a(1-\xi_1)\bigr).\nonumber
\end{eqnarray}\upqed
\end{pf}

We now describe the contours of steepest ascent and descent of $g$.
In Lemma~\ref{bulklemmap}, we analyzed the saddle points of $g$. The
two nonreal saddle points of $g$ are simple (and are conjugate pairs)
and as $g$ is analytic in the upper half plane, the paths of steepest
ascent and descent are the level lines of $\operatorname{Im} g$. These
paths can cross the real line at $-a,0,1/a$. We now describe these paths.

The paths of steepest descent and ascent of the saddle point function
are determined in the upper half plane since the lower half plane is a
reflection. From the saddle point, there are two paths of steepest
ascent, one which goes to $\infty$ while the other ends at $0$. From
the saddle point, there are two paths of steepest descent, one which
ends at $1/a$ and another ending at $-a$. See Figure~\ref
{figcontours1} for an example of the contours of steepest descent and ascent.
%
\begin{figure}

\includegraphics{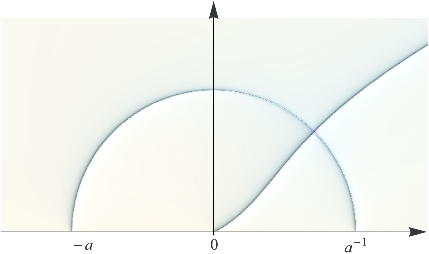}

\caption{A relief plot of $\log|\operatorname{Im}(g(z;\xi)-g(z_\xi;\xi
))|$ with $a=1$, $z_\xi=e^{i \pi/4}$ and $\xi=(1/2,\frac{1}{4}
(2+\sqrt{2} ))$. The relief plot captures where $\operatorname{Im}g(z;\xi)$ is constant and the logarithm is for visual
purposes---it sharpens the relief plot.
The contour of steepest descent starts at the origin and ends at
infinity (goes to the right). The steepest ascent contour starts at
$a^{-1}$ and ends at $-a$. The plot is symmetric in the lower half-plane.}\label{figcontours1}
\end{figure}

We now prove Theorem~\ref{localGibbsmeasure}.

\begin{pf*}{Proof of Theorem~\ref{localGibbsmeasure}}
Below, we will neglect the integer part in the expressions since they
are unimportant.
We will only consider the case $x_1<y_1+1$ which means that
$K^{-1}(x,y)=f_1(x,y)$ for $x\in\mathtt{W}$ and $b \in\mathtt{B}$.
This is due to the following: for $x_1\geq y_1+1$,
%
\begin{eqnarray}
\label{Gibbsothercase}
&& K^{-1}(x,y)\nonumber
\\
&&\qquad  = f_1(x,y)-f_2(x,y)
\nonumber\\[-8pt]\\[-8pt]
&&\qquad  =\frac{(-1)^{(x_1+x_2+y_1+y_2)/4}}{(2\pi i)^2} \nonumber
\\
&&\quad\qquad{}\times \int_{\gamma
_1} \int_{\gamma_2}
dw \,dz\, \frac{ (a+z)^{x_2/2} (a
z-1)^{(2n-x_2)/2} w^{(y_1)/2}}{ z^{(x_1+1)/2}
(w-z) (a+w)^{(y_2+1)/2} (a w-1)^{(2n+1-y_2)/2}},\nonumber\hspace*{-8pt}
\end{eqnarray}
where $\gamma_2$ is a positively oriented closed contour containing
$-a$ and $0$ but not $1/a$, and $\gamma_1$ is a positively oriented
closed contour containing $\gamma_2$ but not $1/a$. The above formula
is found by first moving the contour $\mathcal{E}_2$ in the definition
of $f_1(x,y)$ to the contour $\gamma_2$, followed by moving the
$\mathcal{E}_1$ through $\gamma_2$ to the contour $\gamma_1$; see
Figure~\ref{figcontourswitch}. This picks up a single integral
contribution from $z=w$ (which is negative) with contour of integration
around $\gamma_1$. This single integral contribution is equal to
$f_2(x,y)$, which is seen by deforming the contour through infinity,
which cancels with the term $-f_2(x,y)$ obtained from the split
definition of $K^{-1}(x,y)$ for $x_1 \geq y_1+1$. Since the double
contour integral formula in equation~(\ref{Gibbsothercase}) is
similar to the double contour integral formula in $f_1(x,y)$, the
computation of $K^{-1}(x,y)$ for $x_1>y_1+1$ is similar to the
computation of $K^{-1}(x,y)$ for $x_1<y_1+1$.

We have $(x_1,x_2)=( [2\xi_1 n]+2 \alpha_1+1, [ 2 \xi_2n]+2 \alpha
_2)$ and $(y_1,y_2)=( [2\xi_1 n]+2 \beta_1, [ 2 \xi_2n]+2 \beta
_2+1)$ so that $(x_1/(2n),x_2/(2n)) \to(\xi_1,\xi_2) \in[0,1]^2$ as
$n$ tends to infinity. We also have
%
\begin{equation}
\qquad (-1)^{(x_1+x_2+y_1+y_2)/4} = (-1)^{(x_1+x_2-y_1-y_2+2)/4} =
i (-1)^{(x_1+x_2-y_1-y_2)/4}
\end{equation}
which follows from the fact that if $y_1+y_2 \operatorname{mod}4 =2
\varepsilon+1$, then $2-y_1+y_2 \operatorname{mod}4=2 \varepsilon+1$
for $\varepsilon\in\{0,1\}$. Since
%
\begin{equation}
(-1)^{(x_1+x_2+y_1+y_2)/4}= i(-1)^{(\alpha_1 +\alpha_2-
\beta_1 -\beta_2)/2},
\end{equation}
we find that
%
\begin{eqnarray}
\qquad && f_1\bigl((x_1,x_2),(y_1,y_2)
\bigr)\nonumber
\\
&&\qquad  =\frac{ i(-1)^{(\alpha_1 +\alpha_2-
\beta_1 -\beta_2)/2}}{(2 \pi i)^2}
\\
&&\quad\qquad{}\times  \int_{\mathcal{E}_2} \int
_{\mathcal{E}_1} dz \,dw\,\frac{ e^{n(g(z;\xi)-g(w;\xi))}}{w-z} \frac{ w^{\beta
_1}(a+z)^{\alpha_2}(az-1)^{-\alpha_2}}{z^{\alpha_1+1} (a+w)^{\beta
_2+1}(a w-1)^{-\beta_2}},
\nonumber
\end{eqnarray}
where $g$ is given in~(\ref{bulkdefg}).

By a computation, we have that
$|z_{\xi}| =r_1$
and
%
\begin{equation}
\biggl\llvert\frac{a+z_{\xi}}{a z_{\xi}-1} \biggr\rrvert=r_2,
\end{equation}
where $r_1$ and $r_2$ are defined in~(\ref{bulkdefr}).

From Lemma~\ref{bulklemmap}, we have that for $(\xi_1,\xi_2) \in
\mathcal{D}_c$, then $z_{\xi} \in\mathbb{H}$. By knowing the
contours of steepest ascent and descent for $g$ which are given above,
we deform the contours $\mathcal{E}_1$ and $\mathcal{E}_2$
accordingly which is the same contour deformation as given in Section~4
of~\cite{Joh05}. Explicitly, we move the contour $\mathcal{E}_2$ to
go through $\bar{z}_{\xi}$ and $z_{\xi}$ passing through the
origin and going to infinity. We move $\mathcal{E}_1$ to pass through
$z_{\xi}$ and $\bar{z}_{\xi}$ which passes either side of the
origin of the $x$ axis at $-a$ and $1/a$. Since the contours must cross
under this deformation,
we pick up an additional single integral from the contribution at $z=w$
along the line $z_{\xi}$ to $\bar{z}_{\xi}$. The double contour
integral, whose contours of integration are as given above, is
$O(n^{-1/2})$; see \cite{OR03}, for example. Note that due to our
formulas and orientations of the contours, the additional single
integral term comes with a minus sign (due to our formulas). We find
that $f_1(x,y)$ is given by
%
\begin{eqnarray}
&& -\frac{i (-1)^{(\alpha_1+\alpha_2-\beta_1 - \beta_2)/2}}{2\pi i} \int
_{\bar{z}_{\xi}}^{z_{\xi}}
z^{\beta_1-\alpha_1-1} \frac
{1}{a+z} \biggl( \frac{a+z}{a z-1}
\biggr)^{ \alpha_2-\beta_2} \,dz
\nonumber\\[-8pt]\\[-8pt]
&&\qquad {} +O\bigl(n^{-1/2}\bigr).\nonumber
\end{eqnarray}
We make the change of variables $w=t(z)=(a+z)/(az-1)$. With this change
of variables $t(w)=z$ and $dz= -(1+a^2)/(aw-1)^2 \,dw$. We obtain
%
\begin{eqnarray}
\label{gibbsmeasure-1}
f_1(x,y)& =&-\frac{i (-1)^{(\alpha_1+\alpha_2-\beta
_1 - \beta_2)/2}}{2\pi i}\nonumber
\\
&&\hspace*{6pt}{}\times  \int
_{t(\bar{z}_{\xi} )}^{t(z_{\xi})} t(w)^{\beta_1- \alpha_1 -1}
\frac{ w^{\alpha_2-\beta_2}}{ a+((a+w)/(aw-1))}\frac{ (-(1+a^2))}{(aw-1)^2} \,dw\nonumber
\\
&&{} +O\bigl(n^{-1/2}\bigr)
\\
& =&\frac{i (-1)^{(\alpha_1+\alpha_2-\beta_1 - \beta_2)/2}}{2\pi i} \int
_{t(\bar{z}_{\xi} )}^{t(z_{\xi})}
t(w)^{\beta_1- \alpha
_1 -1} \frac{ w^{\alpha_2-\beta_2-1}}{aw-1} \,dw\nonumber
\\
&&{} +O\bigl(n^{-1/2}\bigr).\nonumber
\end{eqnarray}
We have that $|a+w|^2 \leq r_1^2 |a w-1|^2$ for $w=r_2 e^{i \theta}$
if and only if
%
\begin{eqnarray}
&& a^2+ r_2^2\cos^2 \theta+2
ar_2 \cos\theta+r_2^2 \sin^2
\theta
\nonumber\\[-8pt]\\[-8pt]
&&\qquad \leq r_1^2 \bigl(a^2
r_2^2 \cos^2 \theta-2 a r_2 \cos
\theta+1 +a^2 r_2^2 \sin^2 \theta
\bigr)\nonumber
\end{eqnarray}
which means that
%
\begin{equation}
\cos\theta\leq\frac{a^2 (r_1^2 r_2^2-1) +r_1^2 -r_2^2}{2 r_2
a(1+r_1^2)} = \frac{ \xi_2-\xi_1 +a^2 (\xi_1+\xi_2-1)}{2 a \sqrt{
\xi_1(1-\xi_1)}}.
\end{equation}
The above equation holds with equality if $\theta=\theta_{\xi}$
where $\theta_{\xi} = \arg z_{\xi}$ since
%
\begin{equation}
\operatorname{Re} (z_\xi) = \frac{\xi_2-\xi_1+a^2(\xi_1+\xi_2-1)}{2
a(1-\xi_1)}
\end{equation}
and $|z_\xi|=\sqrt{\xi_1/(1-\xi_1)}$.

Using the residue formula, we write
%
\begin{eqnarray}
\label{gibbsmeasure} \lim_{n \to\infty} f_1(x,y) & =&
\frac{i (-1)^{(\alpha_1+\alpha
_2-\beta_1 - \beta_2)/2}}{(2\pi i)^2}\nonumber
\\
&&{}\times \int_{|w|=r_2} \int_{|z|=r_1}
\frac{ z^{\beta_1-\alpha_1-1} w^{\alpha_2-\beta_2-1}}{( z-((a+w)/(aw-1)))(a w-1)} \,dz \,dw
\nonumber
\\[-8pt]
\\[-8pt]
& =&\frac{i (-1)^{(\alpha_1+\alpha_2-\beta_1 - \beta_2)/2}}{(2\pi
i)^2}\nonumber
\\
&&{}\times \int_{|w|=r_2} \int_{|z|=r_1}
\frac{ z^{\beta_1-\alpha_1-1}
w^{\alpha_2-\beta_2-1}}{a z w-z-a-w} \,dz \,dw.
\nonumber
\end{eqnarray}
Take the change of variables $z \mapsto i/ z$ and $w \mapsto-i/ w$
which gives
%
\begin{eqnarray}
&& \lim_{n \to\infty} f_1(x,y)
\nonumber\\[-8pt]\\[-10pt]
&&\qquad  =\frac{i}{(2\pi i)^2}\int
_{|w|=1/r_2} \int_{|z|=1/r_1} \frac{ z^{\alpha_1-\beta_1-1} w^{\beta
_2-\alpha
_2-1}}{a/( z w)-i/z-a+i/w} \,dz
\,dw.\nonumber
\end{eqnarray}
The above formula, under the change of variables $z \mapsto z r_1$ and
$w \mapsto w r_2$, is equal to
%
\begin{equation}
\frac{r_1^{-\alpha_1+\beta_1} r_2^{-\beta_2+\alpha_2}}{ (2 \pi
i)^2} \int_{|w|=1} \int_{|z|=1}
\frac{ z^{\alpha_1-\beta_1-1}
w^{\beta_2-\alpha_2-1}}{P(z r_1,w r_2)} \,dz \,dw,
\end{equation}
which is equal to
%
\begin{equation}
r_1^{\alpha_1 +\beta_1} r_2^{-\beta_2+\alpha_2}
K^{-1}_{\mu} \bigl((2\alpha_1+1,2
\alpha_2),(2\beta_1,2 \beta_2+1)\bigr).
\end{equation}
Computing the probability of any cylinder event using the above formula
in~(\ref{localstats}) is equivalent to the probability of any cylinder
event using~(\ref{Gibbskast}) in~(\ref{localstats}) which means we
have verified Theorem~\ref{localGibbsmeasure} for $x_1<y_1+1$. As
mentioned above, a similar argument holds for $x_1\geq y_1+1$, but
uses~(\ref{Gibbsothercase}) instead.
\end{pf*}

\section{Discussion on height fluctuations}\label{sec6} \label{conclusion}

In this article, we have focused on studying domino tilings of the
Aztec diamond using the information obtained from the inverse Kasteleyn
matrix. Here, we briefly discuss the height function associated to
domino tilings of the Aztec diamond and its fluctuations in the scaling limit.

The height function, introduced in~\cite{Thu90}, is defined on the
faces of the Aztec diamond graph as follows: the height change between
two adjacent faces is $\pm3$ if there is a dimer covering the shared
edge between the two faces and $\mp1$ otherwise. As we traverse
between two adjacent faces, we choose the sign convention to be $+3$
(or, resp., $-$3) if the left vertex is black (or, resp., white). The
definition is consistent around each vertex, that is, the total height
change around each vertex is zero. Each dimer covering is in bijection
(up to a chosen height level) with the height function; see Figure~\ref
{figprettypicturesofaztecdiamonds} for an example domino tiling
and height function.

Denote $h_n(\mathtt{f})$ to be the height function at a face $\mathtt
{f}$ in the Aztec diamond graph.
Using either the inverse Kasteleyn matrix or the correlation kernel for
the red--blue particles, we can compute the moments of height function
at faces $\mathtt f_1,\ldots,\mathtt f_m$ (i.e., ${\mathbb{E}}[\prod
_{i=1}^m h_n(\mathtt{f_i})]$).

The Gaussian free field $F$ on $\mathbb{H}$, the upper half plane, is
a probability measure on the set of generalized functions on $\mathbb
{H}$ such that for any compactly supported test functions $\phi_1,\phi
_2$, $ \langle F,\phi_1 \rangle:= \int_{\mathbb{H}}
F(z) \phi_1 (z) |dz|^2$ is a real Gaussian random variable with mean
zero and covariance
%
\begin{equation}
{\mathbb{E}}\bigl[ \langle F,\phi_1 \rangle\langle F,
\phi_2 \rangle\bigr] = \int_{\mathbb{H}^2}
\phi_1(z_1) \phi_2 (z_2)
G(z_1,z_2) |dz_1|^2
|dz_2|^2,
\end{equation}
where
%
\begin{equation}
\label{Greensfunction} G(z_1,z_2)=-\frac{1}{2\pi} \log
\biggl\llvert\frac{ z_1-z_2}{z_1
-\bar{z}_2} \biggr\rrvert.
\end{equation}
Let $\widetilde{H}_n(\mathtt f/(2n))=h_n(\mathtt f)-h_n^a(\mathtt f)$
where $h_n^a(\mathtt f)$ is the average height function at the face
$\mathtt f=(\mathtt f_1,\mathtt f_2)$ with $\mathtt f/(2n)=(\mathtt
f_1/(2n),\mathtt f_2/(2n))$. For $\xi=(\xi_1,\xi_2) \in\mathcal
{D}$, we define the map $\Omega\dvtx \mathcal{D}\to\mathbb{H}$ by
%
\begin{eqnarray}\label{intromap}
\Omega(\xi)&=&
\bigl( a^2 \xi_1-\xi_1 + \xi_2 + a^2 \xi
_2-a^2
\nonumber\\[-8pt]\\[-8pt]
&&\hspace*{3pt}{} +i \sqrt
{4a^2(1-\xi_1) \xi_1- \bigl(\xi_2-\xi_1 +a^2 (\xi_2+\xi_1-1)\bigr)^2}\bigr)
/\bigl(2 a(1-\xi_1)\bigr),\hspace*{-27pt}\nonumber
\end{eqnarray}
which is obtained in the proof of Lemma~\ref{bulklemmap}.

We expect that for fixed $0<a<\infty$ and for $\mathtt f \in2 n
\mathcal{D}$, $\sqrt{\pi}\widetilde{H}_n(\mathtt f/(2n))$ converges
weakly to the $\Omega$-pullback of the Gaussian free field $F$ on
$\mathbb{H}$ in the sense
%
\begin{equation}
\frac{\sqrt{\pi}}{n^2}\mathop{\sum_{\mathtt f \in2n\mathcal
{D}}}_{\mathrm{faces}}
\phi\bigl(\mathtt f/(2n)\bigr) \widetilde{H}_n(\mathtt f) \stackrel{
\mathrm{weakly}} {\longrightarrow}\int_{\mathbb{H}} \phi\bigl(
\Omega^{-1}(z) \bigr) J(z) F(z) |dz|^2,
\end{equation}
where $J(z)$ is the Jacobian under the change of variables from
$z=\Omega(\xi)$ to $\xi$, and the sum is over faces $\mathtt{f}$ in
$2 n \mathcal{D}$.

We will not prove this but will mention the possible steps that one
would need to formulate a proof which is based on~\cite{BF08} where
the authors give a complete proof for the height fluctuations of a
certain lozenge tiling model under the pullback of a certain map.
First, one would require the following technical estimates based
on~\cite{BF08}, Section~6:
\begin{longlist}[(5)]
\item[(1)] Compute $K^{-1}(x,y)$ both when $x$ and $y$ are in the bulk
and are asymptotically distant where $x$ is in the bulk\vspace*{2pt} if $\inf_{\xi
\in\partial\mathcal{D}} |x-2n \xi|_1 > n^{2/3}$, and $x$ and $y$
are asymptotically distant if $|x-y|_1 > n^{1/2+\delta}$ for all
$\delta>0$.
\item[(2)] Bound $K^{-1}(x,y)$ both when $x$ and $y$ are
asymptotically close, where $x$ and $y$ are asymptotically close if
$|x-y|_1 \leq n^{1/2+\delta}$.
\item[(3)] Bound $K^{-1}(x,y)$ when either $x$ or $y$, or both, are
close to the edge, where $x$ is close to the edge if $n^{1/3}< \inf_{\xi
\in\partial\mathcal{D}} |x-2n \xi|_1 \leq n^{2/3}$, and $x$
is either in the unfrozen or frozen regions.
\item[(4)] Bound $K^{-1}(x,y)$ when either $x$ or $y$, or both, are at
the edge, where $x$ is at the edge if $ \inf_{\xi\in\partial
\mathcal{D}} |x-2n \xi|_1 \leq n^{1/3}$, and $x$ is either in the
unfrozen or unfrozen regions.
\item[(5)] Bound $K^{-1}(x,y)$ when either $x$ or $y$, or both, are in
the frozen regions.
\end{longlist}

After these estimates are found, one could then use the fact that
moments of the height function can be expressed in terms of the inverse
Kasteleyn matrix~\cite{Bou07a,deTil07,Ken01,Ken00,Ken08b}.
With the above bounds, one could hopefully show that the moment formula
for the height function tends to the moments of a Gaussian random
variable with variance given by~(\ref{Greensfunction}) for
asymptotically distant points in the bulk.
After this, one would need an analogous result to~\cite{BF08}, Theorem~1.2, which shows that the variance of the height
function in the
unfrozen region is order $\log n$. Using a result of this form combined
with the convergence of moments, one could then conclude the proof as
in~\cite{BF08}, Section~5.5. See also~\cite{Pet12b}.

Other approaches for proving fluctuations of the height function
arising from tiling models have been considered in \cite{Dui11},
where the author uses a linear statistic to bypass the rather technical
estimates arising from the frozen--unfrozen boundaries and \cite
{Dub11}, where the author considers the
characteristic function of the height function using the
Cauchy--Riemann operators. The proofs of the results of \mbox{\cite{Ken00,Ken01}} do not apply to domino tilings on the Aztec diamond due to the
domino tilings studied there had the so-called \emph{Temperley boundary
conditions}.

\section*{Acknowledgments} We would like to thank Alexei Borodin,
Maurice Duits, Harald Helfgott, Richard Kenyon, Tony Metcalfe, Jim
Propp and David Wilson for enlightening conversations. We would like to
thank MSRI (Berkeley) where part of this work was carried out and the
Knut and Alice Wallenberg foundation for financial support. We are also very grateful for the anonymous referees
whose comments helped in dramatically improving this paper.

%
%



\printaddresses
\end{document}